\documentclass[12pt,a4paper]{article}
\usepackage{fullpage}
\usepackage{amsmath}
\usepackage{amsthm}
\usepackage{amssymb}
\usepackage{graphicx}
\usepackage{mathtools}
\usepackage{fca}
\usepackage[utf8]{inputenc}
\usepackage{tikz}
\usetikzlibrary{arrows,decorations.pathmorphing,backgrounds,positioning,fit,petri}
\usetikzlibrary{decorations.markings}
\usetikzlibrary{topaths,calc}
\newtheorem*{definition}{Definition}

\newtheorem{proposition}{Proposition}
\newtheorem{theorem}{Theorem}
\newtheorem{claim}{Claim}
\newtheorem{corollary}{Corollary}
\newtheorem{lemma}{Lemma}
\newtheorem{question}{Question}
\newtheorem*{conjecture}{Conjecture}
\DeclareGraphicsExtensions{.pdf,.png,.jpg}

\newcommand{\opgm}{\ensuremath{op^{g,m}}}
\newcommand{\opgmk}{\ensuremath{op^{g,m}(\context)}}
\def\B{\BV}
\def\BO{\mathop{\mbox{\rm B}}}
\title{Rich subcontexts}
\author{Alexandre Albano}
\date{as of July 30}

\begin{document}

\pagenumbering{gobble}
\bibliographystyle{alpha}

% TODO: parametrize \overline{\chi}

\newcommand{\cala}{\ensuremath{\mathcal{A}}}
\newcommand{\calaumpre}{\ensuremath{\mathcal{A}_1^{\chi \neq R}}}
\newcommand{\calaum}{\ensuremath{\mathcal{A}_1}}
\newcommand{\caladoispre}{\ensuremath{\mathcal{A}_2^{\chi \neq R}}}
\newcommand{\caladois}{\ensuremath{\mathcal{A}_2}}
\newcommand{\calanr}{\ensuremath{\mathcal{A}^{\chi \neq R}}}
\newcommand{\calar}{\ensuremath{\mathcal{A}^{\chi = R}}}
\newcommand{\calb}{\ensuremath{\mathcal{B}}}
\newcommand{\calc}{\ensuremath{\mathcal{C}}}
\newcommand{\cald}{\ensuremath{\mathcal{D}}}
\newcommand{\calcnr}{\ensuremath{\mathcal{C}^{\neg R}}}
\newcommand{\calcr}{\ensuremath{\mathcal{C}^R}}
\newcommand{\cale}{\ensuremath{\mathcal{E}}}
\newcommand{\calk}{\ensuremath{\mathcal{K}}}
\newcommand{\caln}{\ensuremath{\mathcal{N}}}
\newcommand{\calr}{\ensuremath{\mathcal{R}}}
\newcommand{\cals}{\ensuremath{\mathcal{S}}}

\newcommand{\calcf}{\ensuremath{\mathcal{K}^*}}

\newcommand{\contextl}{\ensuremath{\mathbb{L}}}

\newcommand*\circled[1]{\tikz[baseline=(char.base)]{
            \node[draw=black, shape=circle,inner sep=0mm, minimum width=3.5mm, line width=0.4pt] (char) {\hspace{0.4mm}\scriptsize \!#1};}}

\newcommand\mic{m^{\circled{$I$}}}
\newcommand\hic{h^{\circled{$I$}}}
\newcommand\gic{g^{\circled{$I$}}}
\newcommand\hjc{h^{\circled{$J$}}}
\newcommand\gjc{g^{\circled{$J$}}}

\newcommand{\cont}{\ensuremath{\mathfrak{C}}}
\def\CN#1{\mathbb{N}^c{#1}}

\maketitle

\section{Introduction}
It is not well understood which effects small changes to a formal
context have on its concept lattice (for the definition of said structure, as well as basics on Formal Concept Analysis, we refer the reader to~\cite{GW1999}). Deleting an object (or an
attribute) may, for example, reduce the number of concepts by up to
50\%. What happens if we delete both, an object and an attribute? The
extreme case is that the number of concepts divides by four. This
indeed happens, for instance, if one deletes an incident object-attribute pair in a contranominal scale. 
Even with the restriction that the deleted pair be non-incident, it is possible 
that the number of concepts gets reduced to one third of the original number. In this paper, we show that there always 
exists an object-attribute pair such that the number of concepts at most halves after its deletion. 
As an application, this result establishes a weak form of a conjecture involving number of concepts 
of a formal context and contranominal scales found as subcontexts.

\section{Motivation}
%The following subsections present two aspects of what motivated to solve the mentioned problem. 
%We show two aspects of what motivates our result %not excluding the possibility that one of them overlaps with the other.
In what follows, we present two aspects of our main motivation to investigate the question. A few elementary defininitions are introduced and related work is shown as well.

\subsection{Local changes to a formal context and their effects}

Understanding how a lattice changes after some part of its associated context gets modified is of great interest to comprehend how a conceptual system evolves. Quite a few problems of this nature were already posed and solved, but many questions still remain. Solutions to those kind of problems may, for instance, provide insights for lattice drawing algorithms.

Removing an object and a non-incident attribute can lead to a loss of more than 50\% of the concepts. Indeed, consider the standard context of the three element chain (that is, the unique two-by-two reduced Ferrers context). It is clear that the removal of its object which has no attributes, along with the removal of the empty column results in a one-by-one full context, which has only one concept.

For a less trivial example which results even in reduced subcontexts, consider the formal context present in Figure~\ref{fig:context}. Its lattice has $15$ elements and, when object $g$ and attribute $m$ both are removed, a (reduced) sub-context with only 7 concepts remains. In contrast, removing $h$ and $m$ results in a sub-context with 9 concepts (which is reduced as well). 

\begin{figure}[htbp]
  \centering
  \begin{cxt}
    \cxtName{$\context$}                                                   
    \att{$m$}                                                      
    \att{$n$}                                                                  
    \att{$o$}                                                        
    \att{$p$}
    \att{$q$}
    \obj{.xx.x}{$g$}
    \obj{.x.xx}{$h$}
    \obj{xx.x.}{$i$}
    \obj{x.xx.}{$j$}
    \obj{x...x}{$k$}
  \end{cxt} 
  \caption{A formal context}
  \label{fig:context}
\end{figure}

So let $\context:=(G,M,I)$ be finite, and let us call a sub-context 
$$\context_{-g-m}:=(G\setminus\{g\},M\setminus\{m\},I\cap((G\setminus\{g\})
\times(M\setminus\{m\}))$$
\textbf{rich}, if $$|\B(\context_{-g-m})|\ge\frac{1}{2}\cdot|\B(\context)|.$$
%It seems somewhat surprising that in Figure~\ref{fig:context} the
%sub-context $\context_{-h-m}$ is rich, while  $\context_{-g-m}$ is
%not.  

As a first contribution, we ask

\begin{question}\label{question_rich}
  Does there always exist a choice of a non-incident object/attribute
  pair $(g,m)$   in a non-trivial finite context $\GMI$, such that
  $\context_{-g-m}$ is rich? 
\end{question}

%Question~\ref{question_rich} would be easy to answer if the definition of rich subcontext were less demanding and required one third instead of one half of the original quantity of concepts. Indeed, in that case, any choice of a non-incident pair $(g,m)$ would do:

%\begin{proposition}
%  For every $\context$ and every non-incident pair $(g,m)$: 
%  $$|\B(\context_{-g-m})|\ge\frac{1}{3}\cdot|\B(\context)|.$$
%\end{proposition}
%\begin{proof}
%  Each concept $(A,B)$ of $\B(\context)$ is such that $(A,B), (A,B \setminus \{m\})$ or $(A \setminus \{g\},B)$ is a concept of $\B(\context_{-g-m})$.
%\end{proof}

%The one third factor present in the inequality above is the best possible: this is guaranteed by the already given example of the standard context of the three element chain.

The aforementioned question is closely related to another problem which was studied by the author during his investigations of contranominal scale free contexts with as many concepts as possible:
\begin{question}\label{question_op}
  Does there always exist a choice of an object/attribute pair $(g,m)$
  in a context $\GMI$, such that $(G,M,J)$, where 
  $$J = I \cup \{(g,n) \mid n \in M \setminus \{m\}\} \cup \{(h,m)
  \mid h \in G \setminus \{g\}\},$$ 
  has at least as many concepts as the original?
\end{question}
It is not surprising that both questions are actually equivalent:
\begin{proposition}
  A subcontext $\context_{-g-m}$ is rich if and only if $(g,m)$ is a suitable choice to answer Question~\ref{question_op} affirmatively.
\end{proposition}
\begin{proof} Let $(g,m)$ be a non-incident pair. The concept lattice of the context $(G,M,J)$ described in Question~\ref{question_op} is clearly $(\{g\},\{m\},\emptyset) + \context_{-g-m}$, where $+$ denotes direct sum of contexts. The associated concept lattice is therefore the direct product of $\BO(2)$ and $\B(\context_{-g-m})$, which has at least as many elements as $\B(\context)$ whenever $\context_{-g-m}$ is rich. Conversely, if $(G,M,J)$ (as defined in the description of Question~\ref{question_op}) has at least as many concepts as $\context$, then clearly the removal of $g$ from $(G,M,J)$ yields a context with at least half of the concepts of $\context$ and in which $m$ is a full column. The removal of a full column does not change the number of concepts (and neither changes the structure of the lattice).
\end{proof}

We now illustrate in which mathematical context Question~\ref{question_op} came to our attention.

\subsection{Extremal results relating number of concepts and contranominal scales}

A context of the form $(\{1,\ldots,j\},\{1,\ldots,j\},\neq)$ will be called a \emph{contranominal scale of size $j$}. Albano and Chornomaz showed the following result relating number of concepts and contranominal scales, where an operation which they call ``doubling'' was employed:

\begin{theorem}\label{teo_nk_extremal}
  An arbitrary formal context with exactly $n$ objects and no contranominal scale of size $c$ as a subcontext may have up to $\sum_{i=0}^{c-1} {n \choose i}$ concepts (but no more). The associated lattices achieving this bound are precisely consecutive doublings of chains inside boolean lattices.
\end{theorem}

It is natural that further results in this direction try to include the number of attributes as a piece of additional information: How many concepts may a formal context with $n$ objects, $m$ attributes and no contranominal scale of size $c$ have? Is that bound achievable? If so, what can be said about contexts (or associated lattices) which achieve this bound? A first step towards that is the following:

\begin{conjecture} Amongst the contexts with $n$ objects, $m$ attributes and no contranominal scale of size $c \leq \min\{n,m\}+1$, every context with maximum number of concepts has a contranominal scale of size $c-1$ as a subcontext.
\end{conjecture}
%\begin{conjecture} Every context having maximum number of concepts amongst the contexts with $n$ objects, $m$ attributes and no contranominal scale of size $c$ has necessarily a contra\-nominal scale of size $c-1$.
%\end{conjecture}

We show in Section~\ref{section:app} that a positive answer to Question~\ref{question_op} is able to establish, without substantial additional effort, the following weaker form of the conjecture:

\begin{claim}\label{claim_weak_conj} Amongst the contexts with $n$ objects, $m$ attributes and no contranominal scale of size $c \leq \min\{n,m\} + 1$, there exists one context with maximum number of concepts which has a contranominal scale of size $c-1$ as a subcontext.
\end{claim}

The aforementioned result and questions belong to the framework of \emph{extremal combinatorics} and may be rewritten in graph-theoretic language as well. Citing Béla Bollobás: ``Extremal graph theory, in its strictest sense, is a branch of graph theory developed and loved by Hungarians.''~\cite{Bollobas:1978}. It is therefore no surprise that a very important milestone of that area was established by another Hungarian:

\begin{theorem}[Turán, 1941]\label{teo_turan} An arbitrary simple graph with exactly $n$ vertices and no clique of size $r$ may have up to $\lfloor \left(1 - \frac{1}{r-1}\right) \frac{n^2}{2} \rfloor$ edges (but no more). The graphs achieving this bound are precisely the complete, balanced, $(r-1)$-partite graphs. 
\end{theorem}

The nature of the questions answered by results like Theorems~\ref{teo_nk_extremal} and~\ref{teo_turan} are analogous. More abstractly, Bollobás describes in~\cite{Bollobas:1978} the following typical setting in an extremal graph theory problem: 

\paragraph{Setting:} Given a property $\mathcal{P}$ and an invariant $\mu$ for a class $\mathcal{H}$ of graphs, we wish to determine the least value $k$ for which every graph $G$ in $\mathcal{H}$ with $\mu(G) > k$ has property $\mathcal{P}$. The graphs $G$ in $\mathcal{H}$ without property $\mathcal{P}$ and $\mu(G) = k$ are called the \emph{extremal graphs} for the problem. 

Theorem~\ref{teo_nk_extremal} and the mentioned conjecture are instances of the described setting above. In both, $\mathcal{P}$ denotes the property of having a contranominal scale of some given size, while $\mu$ corresponds to the number of concepts of a context.  They differ, however, in what $\mathcal{H}$ gets defined to be: the conjecture asks about the interaction between $\mathcal{P}$ and $\mu$ in a subclass of the class $\mathcal{H}$ treated in Theorem~\ref{teo_nk_extremal}, since that theorem is indifferent about the number of attributes that a context has. At this point we start to digress and shall now turn back to rich subcontexts.  

\section{Existence of rich subcontexts}

%The purpose of this paper is to answer the following, which arose during investigations of contranominal scale-free contexts with as many concepts as possible:

%\begin{question}
%  Does there always exist a choice of an object/attribute pair $(g,m)$ in a context $\GMI$, such that $(G,M,J)$, where
%  $$J = I \cup \{(g,n) \mid n \in M \setminus \{m\}\} \cup \{(h,m) \mid h \in G \setminus \{g\}\},$$
%  has at least as many concepts as the original?
%\end{question}

We set some notation. For a formal context $\context = \GMI$, $g \in G$ and $m \in M$ such that $g\! \notI\! m$, we define ${op}^{g,m}(\context) = (G,M,J)$, where $(G,M,J)$ is defined as in Question~\ref{question_op}. Derivation will be denoted by writing the incidence relation in a superscript, and we only use $(\cdot)^J$ to denote derivation in a context obtained by the operation $op$. To attack Question~\ref{question_op}, we make use of the following:
%The following example indicates that the question might not be trivial:

%\begin{figure}[htbp]
%  \centering
%  \begin{cxt}
%    \cxtName{$\context$}                                                   
%    \att{$m$}                                                      
%    \att{$n$}                                                                  
%    \att{$o$}                                                        
%    \att{$p$}
%    \att{$q$}
%    \obj{.xx.x}{$g$}
%    \obj{.x.xx}{$h$}
%    \obj{xx.x.}{$i$}
%    \obj{x.xx.}{$j$}
%    \obj{x...x}{$k$}
%  \end{cxt} 
%  \caption{A formal context}
%  \label{fig:context}
%\end{figure}
%It can be verified that the contexts $\context$, $\opgmk$ and $op^{h,m}(\context)$ have, respectively, fifteen, fourteen and eighteen concepts. Each one of the three contexts is reduced. To attack the proposed question, we make use of the following:
%It can be verified that $\context$ has fifteen, $\opgmk$ has fourteen and $op^{h,m}(\context)$ has eighteen concepts. To attack the proposed question, we make use of the following:
\begin{definition}[mixed generator]
  Let $R \subseteq G$ be fixed. A set $S \subseteq G$ is called a $R$-mixed generator (of the extent $S^{II}$) if, for every $g \in G$, both implications below hold:
  \begin{flalign*}
    \quad i)\, &g \in (S \cap R) \Rightarrow (S \setminus \{g\})^{II} \neq S^{II} &\\
    \quad ii)\, &g \notin (S \cup R) \Rightarrow (S \cup \{g\})^{II} \neq S^{II}. &
  \end{flalign*}
\end{definition}

Note that $G$-mixed generators are minimal generators and that $\emptyset$-mixed generators are extents. We ocasionally refer to a $R$-mixed generator simply by \emph{mixed generator} or by \emph{mixgen} if there is no possibility for ambiguity. Proposition~\ref{prop-mixgen-extents} describes which mixed generators are extents as well with arbitrary $R$. 

\begin{proposition}\label{prop-mixgen-extents}
Let $\GMI$ be a context and $S$ be a $R$-mixed generator. Then, $S$ is an extent if and only if $(S^{II} \setminus S) \cap R = \emptyset$. 
\end{proposition}
\begin{proof} The direct implication is clear since $S^{II} \setminus S = \emptyset$ whenever $S$ is an extent. For the converse, we prove the contraposition. Suppose that $S$ is not an extent and take $g \in S^{II} \setminus S$. Note that $g \in S^{II}$ implies $(S \cup \{g\})^{II} = S^{II}$. Condition $ii)$ of the definition of mixgens forces, therefore, that $g \in S \cup R$. Because of $g \notin S$, it holds that $g \in R$.
\end{proof}

We are particularly interested in the case when $R$ is the set of objects not having some fixed attribute, that is, when $R = G \setminus m^I$. For this reason, we set the notation $\mic = G \setminus m^I$. Note that, in this case, the set $R$ is precisely the set of objects whose derivation are changed by the operation $op$.

\begin{proposition} \label{equiv_mixgen_extensao_disjunto}
Let $\GMI$ be a context, $R = \mic$ for some attribute $m$ and $S \subseteq G$ be a set with $S \cap R = \emptyset$. Then, $S$ is a $R$-mixed generator if and only if $S$ is an extent.
\end{proposition}
\begin{proof}
  Suppose that $S$ is a mixed generator. Therefore, $(S \cup \{g\}) \neq S^{II}$ for every $g \in G \setminus (S \cup R)$. Moreover, $S \cap R = \emptyset$ and $R = \mic$ clearly imply $(S \cup \{g\})^{II} \neq S^{II}$ for every $g \in R$. Combining both yields $(S \cup \{g\})^{II} \neq S^{II}$ for every $g \in G \setminus S$, i.e., $S$ is an extent. For the converse, suppose that $S$ is an extent. Since $S \cap R = \emptyset$, the set $S$ fulfills trivially condition $i)$ of mixgens. Condition $ii)$ is likewise fulfilled by $S$ because of $(S \cup \{g\})^{II} \neq S^{II}$ for every $g \in G \setminus S \supseteq G \setminus (S \cup R)$.
\end{proof}

An easy but handy fact is the following:
\begin{proposition}\label{unique_mixgen}
A set which is a mixed generator and an extent is always the unique mixed generator of itself.
\end{proposition}
\begin{proof}
  Let $S, T \subseteq G$ be $R$-mixed generators with $T^{II} = S^{II} = S$. Since $S$ is an extent, it follows that $T \subseteq S$. A proper containment $T \cap R \subsetneq S \cap R$ would contradict the fact that $S$ is a mixed generator. Similarly, $T \setminus R \subsetneq S \setminus R$ would imply the existence of an object $g$ with $g \in S \setminus R$ and $g \notin T \setminus R$, yielding $(T \cup \{g\})^{II} = T^{II}$ and contradicting the fact that $T$ is a mixed generator.
\end{proof}

Consider the context present in Figure~\ref{fig:context} and set $R = \mic = \{g,h\}$. (The reader is invited to contemplate Figure~\ref{fig:hypergraph} for an alternative representation of that context). The set $S = \{h,i,j,k\}$ (we omit sometimes braces and commas) may be verified as being a mixed generator, since the removal of any element belonging to $S \cap R$ causes its closure (equivalently, its derivation) to change and there is no element in $G \setminus R$ which can be added to $S$ without changing its closure. Note that $S$ is not an extent. Similarly, one may observe that $ghijk$ is an extent but not a mixed generator. Lastly, $ij$ is both an extent and a mixed generator, whereas the set $hj$ is neither an extent nor a mixed generator.

\begin{figure}[htbp]
  \centering
  \begin{tikzpicture}[x=1mm,y=1mm,scale=1.20]
    \node (m) at (10,20) {};
    \node (n) at (10,0) {};
    \node (o) at (20,15) {};
    \node (p) at (0,15) {};
    \node (q) at (40,16) {};
 
    \node (g) at ($(m.center) + (-10,5)$)  {$g$};
    \node (h) at ($(m.center) + (10,5)$) {$h$};    
    \node (i) at ($(q) + (-10,6)$) {$i$};
    \node (j) at ($(q) + (8,-8)$) {$j$};
    \node (k) [below=7mm of n] {$k$};

    \begin{scope}[fill opacity=0.5]
      % aresta g
      \filldraw[fill=white!100] ($(m)+(-4,0)$) 
      to[out=90,in=90] ($(o) + (5,0)$)
      to[out=270,in=270] ($(m) + (-4,0)$)
      ;
      
      % aresta h
      \filldraw[fill=white!100] ($(m)+(4,0)$) 
      to[out=90,in=90] ($(p) + (-5,0)$)
      to[out=270,in=270] ($(m) + (4,0)$)
      ;
      
      % aresta i
      \filldraw[rounded corners=8pt,fill=white!100] ($(o)+(-4,3)$) -- 
        ($(q) + (3,5)$) --
        ($(q) + (4.5,-2)$) --
        ($(o) + (-2.5,-4)$) -- cycle
        ;
      
      % aresta j

      \filldraw[rounded corners=8pt,fill=white!100] ($(q.center)+(-5,7)$) -- 
        ($(q.center) + (5,7)$) --
        ($(q) + (5,-20)$) -- 
        ($(n) + (-5,-4)$) -- 
        ($(n) + (-5,4)$) -- 
        ($(n) + (25,4)$) -- cycle;

      % aresta k

      \filldraw[rounded corners=8pt, fill=white!100] ($(n) + (-6,-6)$) --
        ($(p) + (-4,3)$) -- 
        ($(o) + (4,3)$) --
        ($(n) + (6,-6)$) -- cycle;      
    \end{scope}    
    % we write their names here to avoid transparency
    \fill (m) node {$m$};
    \fill (n) node {$n$};
    \fill (o) node {$o$};
    \fill (p) node {$p$};
    \fill (q) node {$q$};
  \end{tikzpicture}
  \caption{Representation of the context in Figure~\ref{fig:context} through complements of object-intents.}
  \label{fig:hypergraph}
\end{figure}
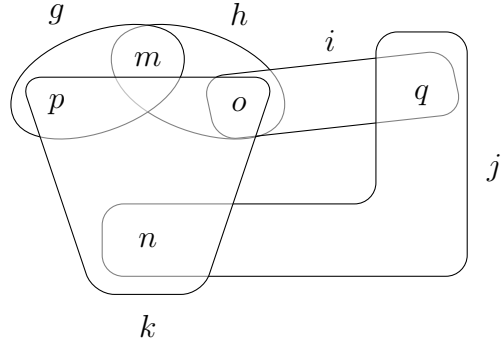
In Figure~\ref{fig:hypergraph}, objects in $R$ correspond to ellipses whereas objects not in $R$ correspond to closed polygonal curves with rounded corners. A mixed generator corresponds to an exact cover of vertices which is minimal with respect to ellipses and maximal with respect to polygonal curves.

The following proposition will be needed later, in the particular case when an object/attribute pair $(g,m)$ ``splits'' the concepts of a context: that is, $g$ is the only object without attribute $m$, while $m$ is the only attribute which $g$ does not have. Similarly with what was done with attributes, we define $\gic = M \setminus g^I$. For other incidence relations we write, for example, $\gjc$.

\begin{proposition} \label{mixgen_extremal} Let $S$ be a mixed generator. If $g \in G \setminus S$ is an object such that ${(S \cup \{g\})^{II} = S^{II} \cup \{g\}}$, $\gic \neq \emptyset$ and $\gic \cap \hic = \emptyset$ for every $h \in S$, then $S \cup \{g\}$ is a mixed generator.
% with ${(S \cup \{g\})^{II} = S^{II} \cup \{g\}}$, then $S \cup \{g\}$ is a mixed generator. 
\end{proposition}
\begin{proof}
  Let $h \in G$ and suppose that $h \notin S \cup \{g\} \cup R$. Because $S$ is a mixed generator and $h \notin (S \cup R)$, it follows that $h \notin S^{II}$ and, since $h \neq g$, we have as well that $h \notin S^{II} \cup \{g\} = (S \cup \{g\})^{II}$, which is equivalent to $(S \cup \{g\} \cup \{h\})^{II} \neq (S \cup \{g\})^{II}$. Now, suppose that $h \in (S \cup \{g\}) \cap R$. If $h = g$, then it is clear that $((S \cup \{g\}) \setminus \{h\})^I = S^I \neq (S \cup \{g\})^I$. Otherwise, we have certainly that $h \in S \cap R$ and, since $S$ is a mixed generator, one necessarily has that there exists an attribute $n \in (S \setminus \{h\})^I$ such that $n \in \hic$ and, as a consequence, it holds that $n \in g^I$. Moreover, $[(S \cup \{g\}) \setminus \{h\}]^I = (S \setminus \{h\})^{I} \cap g^I,$ which implies $n \in [(S \cup \{g\}) \setminus \{h\}]^I$. Thus, $h \notin [(S \cup \{g\}) \setminus \{h\}]^{II}$, which is equivalent to $[(S \cup \{g\}) \setminus \{h\}]^{II} \neq (S \cup \{g\})^{II}$ and establishes condition $ii)$ of the definition.
\end{proof}

Let $\context = \GMI$ be a formal context. A \emph{representative system of mixed generators} is a family of subsets $\cals \subseteq \mathcal{P}(G)$ such that each $S \in \cals$ is a mixed generator and $S \mapsto S^{II}$ is an injection from $\cals$ into $\extents \context$. If the closure mapping is surjective as well, we call $\cals$ a \emph{complete representative system of mixed generators}. For brevity we shall write only \emph{complete system of mixed generators}.

\paragraph{Example:} Setting $R = \mic$ in the context of the Figures~\ref{fig:context} and~\ref{fig:hypergraph} one has that the set family
\begin{equation}\label{example_complete_system}
\cals = \{\emptyset,g,gh,gi,gj,gk,h,hi,hij,hijk,i,ij,ijk,j,k\}
\end{equation}
is a complete system of mixed generators.

For Propositions~\ref{second_prop_preserved},~\ref{mixgen_obtained_by_restriction} and~\ref{carac_not_a_mixgen_inL}, an arbitrary context $\context = \GMI$ and an object/attribute pair $g \in G$, $m \in M$ with $g\! \notI\! m$ are to be considered; furthermore, $\contextl$ denotes $\opgmk$. Proposition~\ref{second_prop_preserved} shows that a mixed generator $S$ in $\context$ is halfway from being a mixed generator in $\contextl$: in that context, $S$ always fulfills condition $ii)$ of the definition.

\begin{proposition} \label{second_prop_preserved}
  Let $R = \mic$ and let $S$ be a $R$-mixed generator in $\context$. Then, for every $h \in G$ with $h \notin S \cup R$ it holds that $(S \cup \{h\})^{JJ} \neq S^{JJ}$, where $J$ denotes derivation in $\contextl$. In particular, if $S\cap R = \emptyset$, then $S$ is a $R$-mixed generator in $\contextl$.
\end{proposition}
\begin{proof}
  Let $h \in G \setminus (S \cup R)$. The fact that $S$ is a mixgen in $\context$ implies $(S \cup \{h\})^{II} \neq S^{II}$. Hence, one has that $(S \cup \{h\})^I \neq S^I$, which is equivalent to $\hic \cap S^I \neq \emptyset$. Now, $h \notin R$ implies $\hic = \hjc$ and this, together with the fact that $S^J \supseteq S^I$, yields $\hjc \cap S^J \neq \emptyset$. That is equivalent to $(S \cup \{h\})^{JJ} \neq S^{JJ}$.
\end{proof}

The reader is maybe aware of the fact that \emph{minimal} generators form a downset: the removal of any element of a minimal generator yields another minimal generator. For mixed generators this is clearly not the case but, as expected, the same phenomenon happens when one removes elements which belong to $R$. Indeed, we have the proposition below. %Note that the condition $R = \mic$ is not required for this fact. 

\begin{proposition} \label{mixgen_obtained_by_restriction}
  For every $R$-mixed generator $S$ and every $T \subseteq S \cap R$, it holds that $S \setminus T$ is a mixed generator. In particular if $R = \mic$, then $S \setminus R$ is a mixed generator in both contexts, $\context$ and $\contextl$. Moreover, $S \setminus R$ is an extent in $\context$ and, therefore, its unique mixed generator in that context. 
\end{proposition}
\begin{proof}
  Set $U = S \setminus T$ so that $S = U \cup T$ and $S^I = U^I \cap T^I$. First we prove that, if the condition $i)$ of mixed generators were not valid for $U$, then it would also not be valid for the set $S$. Suppose, therefore, that there exists $h \in U \cap R \subseteq S \cap R$ with $(U \setminus \{h\})^{I} = U^{I}$. Therefore, we have that $(S \setminus \{h\})^I = [T \cup (U \setminus \{h\})]^I = T^I \cap (U \setminus \{h\})^I = T^I \cap U^I = S^I$. Now, regarding condition $ii)$ of mixed generators, take $h \in G$ with $h \notin (U \cup R)$. Observe that $U \cap R \subseteq S \cap R$ together with $S \setminus R = U \setminus R$, $h \notin R$ and $h \notin U$ imply $h \notin S$. Since $S$ is a mixed generator, we have that $(S \cup \{h\})^I \neq S^I$, which means that there exists an attribute $n \in \hic$ with $n \in S^I \subseteq U^I$. Hence, $(U \cup \{h\})^I \neq U^I$. The three final claims (which require $R = \mic$) come from Propositions~\ref{second_prop_preserved}, ~\ref{equiv_mixgen_extensao_disjunto} and~\ref{unique_mixgen}.
\end{proof}

For a moment, suppose that $\B(\context)$ is finite and let $\cals$ denote a complete system of $\mic$-mixed generators in $\context$. Our goal is to find a sufficient condition for ${|\B(\contextl)| \geq |\B(\context)|}$, and the strategy to arrive at that shall be to reuse as much mixed generators from $\cals$ in $\contextl$ as possible. Aiming this, we establish now what is necessary and sufficient for a mixed generator in $\context$ \emph{not} to be a mixed generator in $\contextl$. As an example, let $\context$ denote the context of Figures~\ref{fig:context} and~\ref{fig:hypergraph} and consider the context $\contextl = op^{g,m}(\context)$, which is depicted in Figure~\ref{fig:hypergraph-l}. Observe that $\{h,i\}^J = \{i\}^J$, which means that $\{h,i\}$ is not\footnote{Because $h \in R$. Notice that the set $R$ is not, in any sense, ``updated''.} a mixed generator in $\contextl$. Moreover, notice that $\{h,i\}$ has the following three properties: first, it does not contain $g$. Second, its intersection with $R$ has cardinality one. Lastly, the derivation in $\context$ (and in $\contextl$) of $\{h,i\} \setminus R = \{h,i\} \setminus \{g,h\} = \{i\}$ equals $\{h,i\}^I \cup \{m\}$. Proposition~\ref{carac_not_a_mixgen_inL} shows that those properties are characteristic.

\begin{figure}[htbp]
  \centering
%  \vspace{2mm}
  \begin{tikzpicture}[x=1mm,y=1mm,scale=1.20]
    \node (m) at (10,20) {};
    \node (n) at (10,0) {};
    \node (o) at (20,15) {};
    \node (p) at (0,15) {};
    \node (q) at (40,16) {};
 
    \node (g) at ($(m.center) + (-3.5,4.5)$)  {$g$};
    \node (h) at ($(o.center) + (0,7.5)$) {$h$};    
    \node (i) at ($(q) + (-10,6)$) {$i$};
    \node (j) at ($(q) + (8,-8)$) {$j$};
    \node (k) [below=7mm of n] {$k$};

    \begin{scope}[fill opacity=0.5]
      % aresta g
      \node[draw=black,fit=(m),inner sep=1ex,circle] (tmp) {};
      
      % aresta h
      \node[draw=black,fit=(o),inner sep=1.65ex,circle] (tmp) {};
      
      % aresta i
      \filldraw[rounded corners=8pt,fill=white!100] ($(o)+(-4,3)$) -- 
        ($(q) + (3,5)$) --
        ($(q) + (4.5,-2)$) --
        ($(o) + (-2.5,-4)$) -- cycle
        ;
      
      % aresta j

      \filldraw[rounded corners=8pt,fill=white!100] ($(q.center)+(-5,7)$) -- 
        ($(q.center) + (5,7)$) --
        ($(q) + (5,-20)$) -- 
        ($(n) + (-5,-4)$) -- 
        ($(n) + (-5,4)$) -- 
        ($(n) + (25,4)$) -- cycle;

      % aresta k

      \filldraw[rounded corners=8pt, fill=white!100] ($(n) + (-6,-6)$) --
        ($(p) + (-4,3)$) -- 
        ($(o) + (4,3)$) --
        ($(n) + (6,-6)$) -- cycle;      
    \end{scope}    
    % we write their names here to avoid transparency
    \fill (m) node {$m$};
    \fill (n) node {$n$};
    \fill (o) node {$o$};
    \fill (p) node {$p$};
    \fill (q) node {$q$};
  \end{tikzpicture}
  \caption{Context $op^{g,m}(\context)$. The operation $op$ corresponds to shrinking one and cutting the other ellipses.}
  \label{fig:hypergraph-l}
\end{figure}
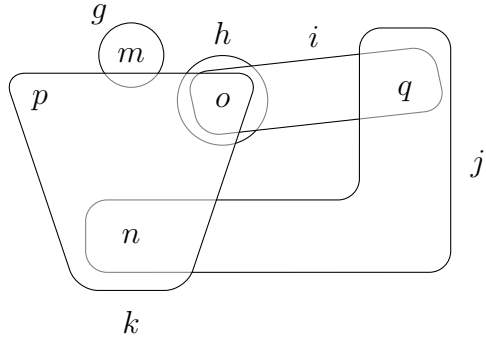

\begin{proposition}\label{carac_not_a_mixgen_inL}
  Let $R = \mic$ and let $S$ be a mixed generator in $\context$. Then, $S$ is not a mixed generator in $\contextl$ if and only if $S \cap R = \{h\}$ and $(S \setminus \{h\})^J = S^I \cup \{m\}$ for exactly one element $h \in R \setminus \{g\}$.
\end{proposition}
\begin{proof}
  Let $S$ be a mixed generator in $\context$. Suppose that $S \cap R = \{h\}$ with $h \neq g$ and that $(S \setminus \{h\})^J = S^I \cup \{m\}$. Since $g \notin S$, it follows from the defintion of $op$ that $S^J = S^I \cup \{m\}$. By transitivity, it holds that $(S \setminus \{h\})^J = S^J$ and $S$ is not a mixgen in $\contextl$. For the converse, by Proposition~\ref{mixgen_obtained_by_restriction}, we have that $S \cap R \neq \emptyset$ and, by Proposition~\ref{second_prop_preserved} it follows that $S$ is not a mixgen in $\contextl$ because it fails to fulfill condition $i)$ of mixgens. That means that there exists $h \in S \cap R$ with $(S \setminus \{h\})^J = S^J$. This shows $S \cap R \supseteq \{h\}$. Note that, since $g$ is the only object without the attribute $m$ in $\contextl$, we have that $h \neq g$ and that $h^J = h^I \cup \{m\}$. From $h \in (S \setminus \{h\})^{JJ}$ follows $(S \setminus \{h\})^J \subseteq h^J = h^I \cup \{m\}$. Moreover, note that $(S \setminus \{h\})^I \subseteq (S \setminus \{h\})^J$ and, by transitivity, $(S \setminus \{h\})^I \subseteq h^I \cup \{m\}$. We now argue that $S \cap R \subseteq \{h\}$. Suppose, by contradiction, that $i \in S \cap R$ with $i \neq h$. Then, $i \in (S \setminus \{h\})$ which implies $m \notin (S \setminus \{h\})^I$. Therefore, $(S \setminus \{h\})^I \subseteq h^I$ which yields $(S \setminus \{h\})^{II} = S^{II}$, contradicting the fact that $S$ is a mixed generator in $\context$. Since $S \cap R = \{h\}$ and $h \neq g$, we have that $S^J = S^I \cup \{m\}$. By transitivity, $S^I \cup \{m\} = (S \setminus \{h\})^J$.
\end{proof}

Keep considering, until Proposition~\ref{prop-semidownset}, an arbitrary context $\context$, an attribute $m \in M$ with $\mic \neq \emptyset$ and suppose that $\cals$ is a complete system of $\mic$-mixed generators in $\context$. Set, for good, $R = \mic$, consider a fixed object $g \in R$ and define $\contextl = \opgmk$. 

We divide $\cals$ in four classes:

$$
\begin{aligned}
  \cals = \caln \cup \cala \cup \calb \cup \calc,
\end{aligned}
$$
where
$$
\begin{aligned}
  \caln &= \{S \in \cals \mid S \mbox{ is not a mixgen in } \contextl \},\\
  \cala &= \{S \in \cals \setminus \caln \mid S^J = S^I\},\\ 
  \calb &= \{S \in \cals \setminus \caln \mid S^J \neq S^I, (S \cup \{g\})^J = S^I\} \mbox{ and }\\ 
  \calc &= \{S \in \cals \setminus \caln \mid S^J \neq S^I, (S \cup \{g\})^J \neq S^I\}.
\end{aligned}
$$

Notice that whenever $S \in \calb$, it holds that $g \notin S$. In contrast, a mixed generator $S \in \calc$ always contains $g$. To see this, it suffices to realize that a subset $T \subseteq G \setminus \{g\}$ always satisfies (exactly) one of the equalities $T^J = T^I$ and $(T \cup \{g\})^J = T^I$.

Mixgens in $\cala$ and in $\calb$ are sufficiently manageable so that we may map them directly to the set of all mixed generators of $\contextl$ and hope that they form a representative system. We will ``rescue'' all the mixed generators in $\caln$ and some in $\calc$ by applying the restriction mapping $res: S \mapsto S \setminus R$. Since $S \setminus R$ is always a mixgen in $\contextl$ (provided $S$ is a mixgen in $\context$), it is easy to realize that $res(\caln) \subseteq \cala$ and $res(\calc) \subseteq \cala$: indeed, in $\context$, the set $S \setminus R$ is the unique mixed generator of itself (cf. Proposition~\ref{mixgen_obtained_by_restriction}), which forces $S \setminus R \in \cals$ whenever $\cals$ is a complete system of mixgens. Also, the equality $(S \setminus R)^I = (S \setminus R)^J$ is obvious.
 %Note that a mixed generator $S \in res(\caln)$ is actually an extent and, therefore, it is the unique mixed generator of $S^{II} = S$. The same can be asserted for an element in $\calar \supseteq res(\calcr)$.

The operation $op$ changes the derivation of objects in $R$. Therefore, it makes sense to devote special attention to mixgens which have non-empty intersection with $R$. In contrast, we need a condition which is stronger than $h \in R \setminus S$ to help us identify mixgens which are largely unaffected by the operation $op$. Let $S \subseteq G$ and $h \in G$. We say that $S$ \emph{strongly avoids $h$} if $S^I \cap (\hic \setminus \{m\}) \neq \emptyset$. That is, $h$ does not belong to $(S^I \setminus \{m\})^I$ and, in particular, neither to $S^{II}$ or to $S$. Further, we define:
$$\chi(S) = \{g \in R \mid S \mbox{ strongly avoids } g \}.$$
The following claim follows directly from the antitone property of the derivation operator:

\begin{proposition}\label{chi_antitono}
  Let $S,T \subseteq G$ with $S \subseteq T$. Then, $\chi(T) \subseteq \chi(S)$.
\end{proposition}

We shall be able to ``rescue'' mixed generators in $\calc$ which contain each element of $R$. Define:
$$
\begin{aligned}
  \calcr = \{S \in \calc \mid R \subseteq S\} \mbox{ and }
  \calcnr = \{S \in \calc \mid R \nsubseteq S\}.
\end{aligned}
$$  

It turns out that the function $\chi$ is able to distinguish the image sets $res(\caln)$ and $res(\calcr)$, as the following two propositions show.

\begin{proposition}
  Let $S \in \calcr$. Then, $\chi(S \setminus R) = R$.
\end{proposition}
\begin{proof}
  Set $T = S \setminus R$ and suppose that $|R| \geq 2$. Let $h \in R$. Note that $m \notin (S \setminus \{h\})^I$. Because $S$ is a mixed generator and $h \in R \subseteq S$, it follows that $(S \setminus \{h\})^I \neq S^I$. Therefore, there exists $n \in (S \setminus \{h\})^I$ such that $n \in \hic$ and $n \neq m$. In particular, $(S \setminus \{h\})^I \cap (\hic \setminus \{m\}) \neq \emptyset$, which is to say that the set $S \setminus \{h\}$ strongly avoids $h$. That is, $h \in \chi(S \setminus \{h\}) \subseteq \chi(S \setminus R)$, where the containment follows from Proposition~\ref{chi_antitono}. 
Since the object $h$ was arbitrary, we have that $\chi(S \setminus R) = R$. For the remaining case, necessarily $R = \{g\}$. The condition $S^J \neq S^I$ allows us to take an attribute $n \in S^J \setminus S^I$. Clearly $n \in \gic$, because $g$ is the only object whose derivation with respect to $J$ and $I$ differ. Similarly, $n \in h^I$ for every $h \in S \setminus \{g\}$ because of $n \in S^J$. Note that $S \in \calcr$ forces $g \in S$ which in turn implies $m \notin S^J$ and, as a consequence, $n \neq m$. Combining those assertions we arrive, in particular, at $(S \setminus \{g\})^I \cap (\gic \setminus \{m\}) \neq \emptyset$, that is, $g \in \chi(S \setminus \{g\})$. 
\end{proof}

\begin{proposition} 
  Let $S \in \caln$. Then, $\chi(S \setminus R) \neq R$.
\end{proposition}
\begin{proof}
  By Proposition~\ref{carac_not_a_mixgen_inL}, it follows that $S \cap R = \{h\}$ with $h \neq g$ and $(S \setminus \{h\})^J = S^I \cup \{m\}$. Of course, $(S \setminus \{h\})^I = (S \setminus \{h\})^J$ and by transitivity, $(S \setminus \{h\})^I = S^I \cup \{m\}$. Thus, the only attribute $n$ satisfying $n \in \hic$ and $n \in (S \setminus \{h\})^I$ is $n = m$. Consequently, the intersection between $\hic \setminus \{m\}$ and $(S \setminus \{h\})^I$ is empty, that is, the set $S \setminus \{h\} = S \setminus R$ does not strongly avoid $h$.
\end{proof}

In order to organize to which subset of $\cala$ the restriction mapping maps to, we partition the class $\cala$ in three:
$$
\begin{aligned}
  \caladoispre &= res(\caln)\\
  \calar &= \{S \in \cala \mid \chi(S) = R\},\\
  \calaumpre &= \{S \in \cala \mid \chi(S) \neq R, S \notin res(\caln)\}.
\end{aligned}
$$
Suppose that $S$ is a mixed generator in $\contextl$. Proposition~\ref{mixgen_extremal} guarantees that $S \cup \{g\}$ is a mixed generator in $\contextl$ as well. Thus, we may consider each mixed generator $S \in \cals \setminus \caln$ in pairs $(S, S \cup \{g\})$. Such pairs potentially collapse, i.e., it could be that $g \in S$ and, to avoid this, we only consider a mixgen $S$ in this pairwise way when certainly $S$ does not contain $g$ already: more specifically, when ${S \in \caladoispre \cup \calar \cup \calb}$. The general situation regarding the seven partition classes we defined, together with the relevant mappings is as retracted in Figure~\ref{fig:decomp_gran_max}. An example of such a decomposition is given in Figure~\ref{fig:exemplo_decomposicao_sistema_completo1}. Notice that such a partition of $\cals$ depends on $\contextl$, which in turn depends on the choice of $m$ and of $g \in \mic$.

\begin{figure}[htbp]
  \centering
  \begin{tikzpicture}
    [x=1mm,y=1mm,classe/.style={rectangle, draw=white!50, inner sep=0pt, minimum width=4mm, minimum height=5mm},
      classegrande/.style={rectangle,draw=white!50, inner sep=0pt, minimum width=18mm, minimum height=5mm}]
    \node at (0,0) (cals) {$\cals = $}; 
    \node[classe] at (7,0) (caln) {$\caln$};
    \node at (13,0) (cup1) {$\cup$}; 
    \node[classe] at (22,0) (calaum) {$\calaumpre$}; 
    \node (cup2) [right=1mm of calaum] {$\cup$}; 
    \node[classegrande] (caladois) [right=2mm of cup2] {$\caladoispre$}; 
    \node (cup3) [right=1mm of caladois] {$\cup$}; 
    \node[classegrande] (calar) [right=2mm of cup3] {$\calar$}; 
    \node (cup4) [right=0mm of calar] {$\cup$}; 
    \node[classegrande] (calb) [right=2mm of cup4] {$\!\!\!\!\calb$}; 
    \node (cup5) [right=-1mm of calb] {$\cup$}; 
    \node[classe] (calcr) [right=0mm of cup5] {$\calcr$}; 
    \node (cup6) [right=0mm of calcr] {$\cup$}; 
    \node[classe] (calcnr) [right=0mm of cup6] {$\calcnr$}; 
    
    \node[classe] (eldom1) [below=5mm of calaum] {$S$};     
    \node (ni1) [below=8mm of calaum.40,rotate=-90] {$\ni$};
    
    \node[classe] (eldom2) [below=5mm of caladois.south west] {$S$};     
    \node (ni2) [below=10mm of caladois.150,rotate=-135] {$\ni$};
    
    \node[classe] (eldom3) [below=5mm of calar.south west] {$S$};     
    \node (ni3) [below=10mm of calar.150,rotate=-135] {$\ni$};
    
    \node[classe] (eldom4) [below=5mm of calb.south west] {$S$};     
    \node (ni4) [below=10mm of calb.150,rotate=-135] {$\ni$};
    
    \node (ph1) [below=25mm of calaum.south] {$S$};

    \node (ph2) [below=27.7mm of caladois.west] {$S$};
    \node (ph3) [below=30mm of caladois.35] {$S \cup \{g\}$};
    
    \node (ph4) [below=27.7mm of calar.west] {$S$};
    \node (ph5) [below=30mm of calar.35] {$S \cup \{g\}$};
    
    \node (ph6) [below=27.7mm of calb.west] {$S$};
    \node (ph7) [below=30mm of calb.35] {$S\cup \{g\}$};
    
    \draw (eldom1) edge[|->] (ph1);

    \draw (eldom2.south) edge[distance=0.6cm,out=270,in=90,|->] (ph2);
    \draw (eldom2.east) edge[distance=0.9cm,out=360,in=90,|->] (ph3.north);
    
    \draw (eldom3.south) edge[distance=0.6cm,out=270,in=90,|->] (ph4);
    \draw (eldom3.east) edge[distance=0.9cm, out=360,in=90,|->] (ph5.north);
    
    \draw (eldom4.south) edge[distance=0.6cm, out=270,in=90,|->] (ph6);
    \draw (eldom4.east) edge[distance=0.9cm, out=360,in=90,|->] (ph7.north);
    
    \draw[rounded corners=8pt,->>] (caln.north) -- ($(caln.north) + (0,7)$) -- ($(caladois.north) + (0,7)$) -- (caladois.north);
    \draw[rounded corners=8pt,right hook->] (calcr.north) -- ($(calcr.north) + (0,7)$) -- ($(calar.north) + (0,7)$) -- (calar.north);
    
    \node (ressurj) [above left=7mm and -1mm of cup2, scale=0.88] {$res$};
    \node (resinj) [above right=7mm and 2mm of cup4, scale=0.88] {$res$};
  \end{tikzpicture}
  \caption{Decomposition of a complete system of $R$-mixed generators with $R = G \setminus m^I$.}
  \label{fig:decomp_gran_max}
\end{figure}
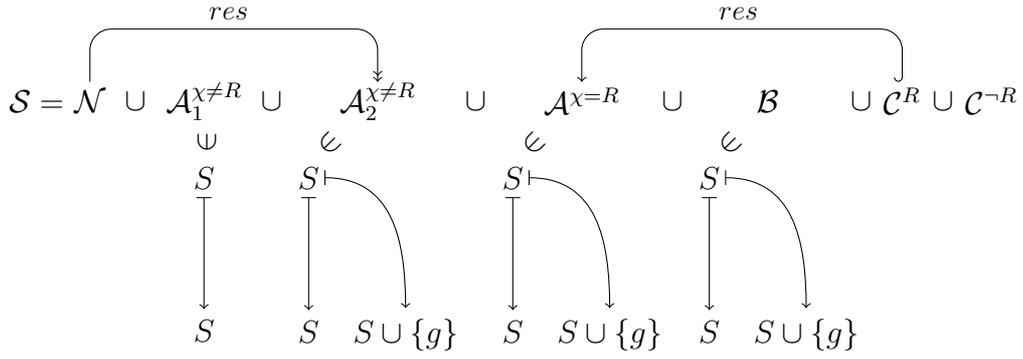

The reader has probably noticed from the characterization of elements in $\caln$ that the derivation $[res(S)]^J$ is not much different from $S^I$ whenever $S \in \caln$. Indeed, they differ only by the presence of the attribute $m$. As a result, the restriction mapping is injective when applied to that class:

\begin{proposition}\label{prop_res_inj}
  The restriction mapping is injective when applied to $\caln$.
\end{proposition}
\begin{proof}
  Let $S,T \in \caln$ and suppose that $S \setminus R = T \setminus R$. By Proposition~\ref{carac_not_a_mixgen_inL}, one has that $|S \cap R| = |T \cap R| = 1$ and, in particular, $m \notin S^I \cup T^I$. Moreover, Proposition~\ref{carac_not_a_mixgen_inL} implies $(S \setminus R)^J = S^I \cup \{m\}$ as well as $(T \setminus R)^J = T^I \cup \{m\}$. From $S \setminus R = T \setminus R$ follows $S^I \cup \{m\} = T^I \cup \{m\}$ and $m \notin S^I \cup T^I$ yields $S^I = T^I$. Since $\caln$ is a representative system of mixgens, this forces $S = T$.
\end{proof}

Consider the context $\context$ depicted earlier in Figures~\ref{fig:context} and~\ref{fig:hypergraph}. Further, consider the complete system of mixed generators $\cals$ given in~\eqref{example_complete_system}. Figure~\ref{fig:exemplo_decomposicao_sistema_completo1} illustrates the partition described in Figure~\ref{fig:decomp_gran_max} with $\contextl = op^{g,m}(\context)$. Note that, in $\context$, $R = \{g,h\}$ and, therefore, when keeping $m$ fixed, we have two (potentially different) decompositions: one for each choice of $\contextl$. To ease the notation we shall, from now on, omit the superscript $\chi \neq R$ and write only $\calaum$ and $\caladois$ whenever we partition $\cala$ in three (that is, $\cala = \calaum \cup \caladois \cup \calar$).

\begin{figure}[htbp]
%  \vspace{2mm}
  \centering
  \begin{tikzpicture}
    [x=1mm,y=1mm,classe/.style={rectangle, thick, draw=black!50, minimum width=10mm, minimum height=18mm}]    
    \node[classe, align=left] (caln) at (0,0) {$hi$\\$hij$\\$hijk$}; 
    \node[classe, align=left, right=5mm of caln] (calaum) {$k$\\$gk$}; 
    \node[classe, align=left, right=5mm of calaum] (caladois) {$i$\\$ij$\\$ijk$}; 
    \node[classe, align=left, right=5mm of caladois] (calar) {$\emptyset$\\$j$}; 
    \node[classe, align=left, right=5mm of calar] (calb) {$h$}; 
    \node[classe, align=left, right=5mm of calb] (calcr) {$gh$}; 
    \node[classe, align=left, right=5mm of calcr] (calcnr) {$g$\\$gi$\\$gj$}; 

    % letreiros
    \node[below=2mm of caln.south] {$\caln$};
    \node[below=2mm of calaum.south] {$\calaum$};
    \node[below=2mm of caladois.south] {$\caladois$};
    \node[below=2mm of calar.south] {$\calar$};
    \node[below=2mm of calb.south] {$\calb$};
    \node[below=2mm of calcr.south] {$\calcr$};
    \node[below=2mm of calcnr.south] {$\calcnr$};

    \draw[rounded corners=8pt, semithick, left hook->>] (caln.north) -- ($(caln.north) + (0,5)$) -- ($(caladois.north) + (0,5)$) -- (caladois.north);
    \draw[rounded corners=8pt, semithick, right hook->] (calcr.north) -- ($(calcr.north) + (0,5)$) -- ($(calar.north) + (0,5)$) -- (calar.north);
    \node[above=5mm of calaum, scale=0.88] (letreirores1)  {$res$};
    \node[above=5mm of calb, scale=0.88] (letreirores2)  {$res$};
  \end{tikzpicture}
%  \vspace{2mm}

  \caption{Decomposition of the complete system of mixed generators given in~\eqref{example_complete_system} with the choice $\contextl = op^{g,m}(\context)$.} 
  \label{fig:exemplo_decomposicao_sistema_completo1}
\end{figure}
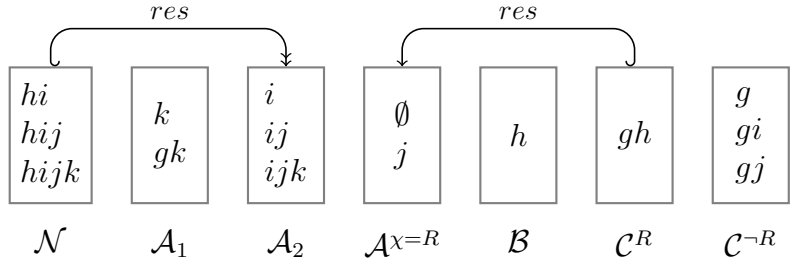

Regarding the same complete system of mixgens but with a different choice of $\contextl$, namely $\contextl = op^{h,m}(\context)$, we have the decomposition depicted in Figure~\ref{fig:exemplo_decomposicao_sistema_completo2}. 

\begin{figure}[htbp]
  \centering
  \begin{tikzpicture}
    [x=1mm,y=1mm,classe/.style={rectangle, draw=black!50, thick, minimum width=10mm, minimum height=18mm}]    
    \node[classe, align=left] (caln) at (0,0) {$gk$}; 
    \node[classe, align=left, right=5mm of caln] (calaum) {$hi,i$\\$hij,ij$\\$hijk,ijk$}; 
    \node[classe, align=left, right=5mm of calaum] (caladois) {$k$}; 
    \node[classe, align=left, right=5mm of caladois] (calar) {$\emptyset$\\$j$}; 
    \node[classe, align=left, right=5mm of calar] (calb) {$g$\\$gi$\\$gj$}; 
    \node[classe, align=left, right=5mm of calb] (calcr) {$gh$}; 
    \node[classe, align=left, right=5mm of calcr] (calcnr) {$h$}; 

    % letreiros
    \node[below=2mm of caln.south] {$\caln$};
    \node[below=2mm of calaum.south] {$\calaum$};
    \node[below=2mm of caladois.south] {$\caladois$};
    \node[below=2mm of calar.south] {$\calar$};
    \node[below=2mm of calb.south] {$\calb$};
    \node[below=2mm of calcr.south] {$\calcr$};
    \node[below=2mm of calcnr.south] {$\calcnr$};

    \draw[rounded corners=8pt, semithick, left hook->>] (caln.north) -- ($(caln.north) + (0,5)$) -- ($(caladois.north) + (0,5)$) -- (caladois.north);
    \draw[rounded corners=8pt, semithick, right hook->] (calcr.north) -- ($(calcr.north) + (0,5)$) -- ($(calar.north) + (0,5)$) -- (calar.north);
    \node[above=5mm of calaum, scale=0.88] (letreirores1)  {$res$};
    \node[above=5mm of calb, scale=0.88] (letreirores2)  {$res$};
  \end{tikzpicture}
  \caption{Decomposition of the complete system of mixed generators given in~\eqref{example_complete_system} with the choice $\contextl = op^{h,m}(\context)$.}
  \label{fig:exemplo_decomposicao_sistema_completo2}
\end{figure}
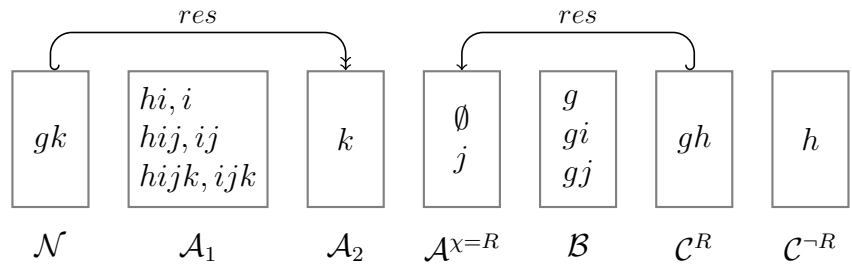

Before we consider the images of the mappings present in Figure~\ref{fig:decomp_gran_max} and establish whether they form a representative system in $\contextl$ or not, we shall turn our attention for a moment to the closure operator in $\contextl$. More specifically, we shall relate it to the closure operator in $\context$ and to the function $\chi$. For that, we define 
$$\overline{\chi}(S) = R \setminus \chi(S).$$
The functions $\chi$ (and therefore $\overline{\chi}$) are to be calculated always regarding the original incidence relation, that is, $I$. 
\begin{proposition}\label{prop_auxiliar_extensao} Let $S \in \caln \cup \cala \cup \calb \cup \calcr$. Then:
$$
  \begin{aligned}
    S \in \calcr &\mbox{ $\Rightarrow$ } S = S^{II} = S^{JJ},\\
    S \in \cala &\mbox{ $\Rightarrow$ } S \subseteq S^{II} \subseteq S^{JJ},\\
    S \in \caln \cup \calb &\mbox{ $\Rightarrow$ } S \subseteq S^{II} \setminus \{g\} \subseteq S^{JJ}.
  \end{aligned}
$$
Furthermore, ${S^{JJ} \setminus S^{II} \subseteq \overline{\chi}(S)}$ and $S^{II} \setminus S \subseteq \overline{\chi}(S)$.
\end{proposition}
\begin{proof}
  Whenever $S \in \calcr$, the set $S$ is a mixgen in $\context$ and $\contextl$ such that $R \subseteq S$, and Proposition~\ref{prop-mixgen-extents} assures that $S$ is an extent in both contexts. If $S \in \cala$ then, clearly, $S^J = S^I$ as well as $S^{JJ} = S^{IJ} \supseteq S^{II}$. Now, suppose that $S \in \caln \cup \calb$. Proposition~\ref{carac_not_a_mixgen_inL} implies that we have $S^J = S^I \cup \{m\}$ in case that $S \in \caln$. The same equality follows easily from the definition of $\calb$, so that that relationship is valid in either case. Then, it follows that $S^{JJ} = (S^I \cup \{m\})^J = S^{IJ} \cap m^J = S^{IJ} \setminus \{g\} \supseteq S^{II} \setminus \{g\}$.

Regarding the last two assertions, we assume $S \notin \calcr$ since otherwise both follow trivially from the already established fact that $S$ is an extent in both contexts. For the second claim: take an object $h \in S^{II} \setminus S$. Since $S$ is a mixed generator in $\context$, we have that $h \in R$ and it is clear that $\hic \cap S^I = \emptyset$ and, consequently, $S$ does not strongly avoid $h$.
Now, let $h \in S^{JJ} \setminus S^{II}$. We may suppose $h \neq g$: indeed, in $\contextl$, the object $g$ is an extremal point of every extent containing it, that is, $g \in S^{JJ}$ implies $g \in S \subseteq S^{II}$. Of course, $h \notin S^{II}$ is equivalent to the condition of $\hic \cap S^I$ being non-empty, whereas $h \in S^{JJ}$ if and only if $\hjc$ and $S^J$ are disjoint. Since $S^J \supseteq S^I$ and $\hjc = \hic$ or $\hjc = \hic \setminus \{m\}$, only the second equality may and must hold, which implies $h \in R$, as well as $\hic \cap S^I = \{m\}$. In particular, the set $S$ does not strongly avoid $h$.
%$S^J = S^I \cup \{m\}$ and $\overline{h^J} = \overline{h^I} \setminus \{m\}$ it follows that $\overline{h^I} \cap S^I = \{m\}$. 
\end{proof}

The following theorem gives information about derivation in $\contextl$ of mixed generators which were mapped from $\cals$ and shows a sufficient condition for $|\B(\contextl)| \geq |\B(\context)|$.

\begin{theorem}\label{teorema1} The mappings $S \xmapsto{\alpha} S$ and $S \xmapsto{\beta} S \cup \{g\}$ are injections, respectively, from $\cala \cup \calb$ and $\caladois \cup \calar \cup \calb$ into the set of all mixed generators of $\contextl$. Their images are disjoint sets whose union is a representative system. The corresponding intents are given by
$$
\begin{aligned}
  [\alpha|_{\cala}(S)]^J &= [\beta|_{\calb}(S)]^J = S^I, \\
  [\alpha|_{\calb}(S)]^J &= S^I \cup \{m\} \mbox{ and }\\
  [\beta|_{\cala}(S)]^J &= S^I \setminus \{m\}, \\
\end{aligned}
$$
where $|$ denotes domain restriction. Intents which already were intents of the original context are characterized via $\alpha(S)^J \in \intents \context \Leftrightarrow S \in \cala$ and ${\beta(S)^J \in \intents \context \Leftrightarrow S \notin \calar}$. The corresponding extents are given by
$$
\begin{aligned}
  [\alpha(S)]^{JJ} &= S \cup \left(\overline{\chi}(S) \setminus \{g\}\right) \mbox{ and }\\
  [\beta(S)]^{JJ} &= S \cup \overline{\chi}(S) \cup \{g\}.
\end{aligned}
$$
In particular, $|\B(\contextl)| \geq |\B(\context)| + |\calb| - |\calcnr|$ whenever $\cals$ is finite.
\end{theorem}
\begin{proof}
The set $\beta(S) = S \cup \{g\}$ is always a mixed generator because of Proposition~\ref{mixgen_extremal}. We show the formulas for the intents. Let $S \in \cala$. Then, $\alpha(S)^J = S^J = S^I$. Suppose that $S \in \caladois \cup \calar$. Then, $\beta(S)^J = (S \cup \{g\})^J = S^J \setminus \{m\} = S^I \setminus \{m\}$. Now, let $S \in \calb$. Consequently, one has $S^J \neq S^I$, $(S \cup \{g\})^J = S^I$ and $m \notin S^I$. The first two facts together with $g^J = M \setminus \{m\}$ imply $S^I \cup \{m\} = S^J = \alpha(S)^J$. Lastly, $\beta(S)^J = (S \cup \{g\})^J = S^J \setminus \{m\} = S^I$. Before proving the properties of the mappings, we prove the two equivalences mentioned.

Let $S \in \calb$. We prove that $\alpha(S)^J$ is not an intent of $\context$. Note that $S \cap R \neq \emptyset$. Suppose, by contradiction, that there exists $T \subseteq G$ with $T^I = \alpha(S)^J = S^J = S^I \cup \{m\}$. In particular, $T^{II} \subseteq S^{II}$, which implies $(S \cup \{h\})^{II} = S^{II}$ for each $h \in T$. Since $m \in T^I$, it follows that $T \cap R = \emptyset$ and this, combined with the fact that $S$ is a mixed generator, yields $T \subseteq S$. Moreover, one has of course that $(S \setminus R)^J \supseteq S^J$ 
%Of course, $(S \setminus R)^J \supseteq S^J = S^I \cup \{m\}$. Also clear is $m \in (S \setminus R)^J$, yielding $(S \setminus R)^J \supseteq S^J$. 
and, on the other hand $T \cap R = \emptyset$ and $T \subseteq S$ imply $(S \setminus R)^J \subseteq T^J = T^I = S^J$. We arrive at $(S \setminus R)^{JJ} = S^{JJ}$, which contradicts the fact that $S$ is a mixed generator in $\contextl$. Now, let $S \in \calar$. We show that $\beta(S)^J$ is not an intent of $\context$. Set $B = \beta(S)^J = S^I \setminus \{m\}$. Since $\chi(S) = R$, it follows that the set $(S^I \setminus \{m\}) \cap \hic = B \cap \hic$ is non-empty for each $h \in R$. Now, for any set $T \subseteq G$, an equality $T^I = B$ requires that $m \notin T^I$, which clearly forces $T^I \cap \hic = \emptyset$ for some $h \in R$ and, therefore, in any case, it holds that $T^I \neq B$.

%Note that $S$, in $\context$, strongly avoids every $h \in R$ and, in particular, $S \cap R = \emptyset$. Therefore, $\beta(S)^J = (S \cup \{g\})^J = S^J \setminus \{m\} = S^I \setminus \{m\}$. Since $S$ strongly avoids each $h \in R$, it holds that $S^I \cap [(\hic) \setminus \{m\}] \neq \emptyset$ and, in particular, $(S^I \setminus \{m\}) \cap [(\hic) \setminus \{m\}] \neq \emptyset$. Therefore, every set $T \subseteq G$ with $T^I = S^I \setminus \{m\}$ must avoid strongly each $h \in R$ and, therefore, must satisfy $T \cap R = \emptyset$, which forces $m \in T^I$. Consequently, $S^I \setminus \{m\}$ is not an intent of $\context$. 

Note that $m \in S^I$ in case that $S \in \caladois \cup \calar$ and, similarly, $m \notin S^I$ whenever $S \in \calb$. Combining this with the shown formulas for the intents, as well as with the fact that the sets $S$ are drawn from a representative system, it follows that each of the restricted, composite mappings $S \mapsto \alpha|_{\cala}(S)^J, S \mapsto \alpha|_{\calb}(S)^J, S \mapsto \beta|_{\cala}(S)^J, S \mapsto \beta|_{\calb}(S)^J$ maps into the set of intents of $\contextl$. We will now show that those four mappings have disjoint image sets. Let $S \in \cala, T \in \calb$. Since $\alpha(S)^J$ is an intent of $\context$ whereas $\alpha(T)^J$ is not, it follows, in particular, that $\alpha(S)^J \neq \alpha(T)^J$. Moreover we have, as shown, $\alpha(S)^J = S^I$ and $\beta(T)^J = T^I$. An equality $\alpha(S)^J = \beta(T)^J$ can not hold since this would force $S^I = T^I$, which is not possible in a representative system. Now, let $S \in \caladois \cup \calar$ and $T \in \calb$. The intent formula implies that $\beta(S)^J \cup \{m\}$ is an intent of $\context$, whereas $\beta(T)^J \cup \{m\} = \alpha(T)^J$ is not. Thus, $\beta(S)^J \neq \beta(T)^J$. Lastly, it is trivial that $m \notin \beta(S)^J$ and $m \in \alpha(T)^J$, causing $\beta(S)^J \neq \alpha(T)^J$. Regarding the extents, we first prove that $\overline{\chi}(S) \setminus \{g\} \subseteq S^{JJ}$. Indeed, for an object $h \in \overline{\chi}(S) \setminus \{g\}$ one has that $S$ does not strongly avoid $h$, that is, $\hic \cap S^I \subseteq \{m\}$. Now, since $h$ is distinct from $g$, it follows that $\hjc = \hic \setminus \{m\}$. Consequently, $\emptyset = \hjc \cap S^I$ and, in any case, the calculated intent $S^J$ is a subset of $S^I \cup \{m\}$, which causes $\hjc \cap S^J = \emptyset$, that is, $h \in S^{JJ}$.  Combining this fact with the following two properties yields the formula for $\alpha(S)^{JJ}$: first, for each $S \in \cals$, one has that $g \in S^{JJ}$ implies $g \in S$, and the converse is obvious. Second, Proposition~\ref{prop_auxiliar_extensao} guarantees that, in general (except when $S \in \calcnr$), $S \setminus \{g\} \subseteq S^{II} \setminus \{g\} \subseteq S^{JJ} \setminus \{g\}$ as well as $S^{JJ} \setminus S^{II} \subseteq \overline{\chi}(S)$ and $S^{II} \setminus S \subseteq \overline{\chi}(S)$. The formula for $\beta(S)^{JJ}$ follows from the one for $\alpha(S)^{JJ}$ and from the fact that $g$ is an extremal point of every extent containing $g$.
\end{proof}

We now investigate what happens to the seven classes depicted in Figure~\ref{fig:decomp_gran_max} when one changes the object $g \in \mic$ chosen to perform the operation $op$. The attribute $m$, however, is to be considered fixed. Note that, for a fixed set $S$, the function value $\chi(S)$ depends on $m$ but does not depend on $\contextl = op^{g,m}(\context)$. 

\begin{lemma}[stability]\label{lemma_stability}
  The classes $\calar$ and $\calcr$ do not depend on the choice of $\contextl$. Moreover, let $S \in \cals$. Then,
  $$
  \begin{aligned}
    S \in \caln \cup \cala &\Leftrightarrow \chi(S) = \chi(S \setminus R) \\
    S \in \calb \cup \calc &\Leftrightarrow \chi(S) \subsetneq \chi(S \setminus R).
  \end{aligned}
  $$
  In particular, $\calb \cup \calcnr$ does not depend on the choice of $\contextl$.
\end{lemma}
\begin{proof} The claim regarding $\calar$ is clear because $\chi(S) = R$ implies $S \cap R = \emptyset$, which means that $S^J$ is always $S^I$, independently of the choice of $\contextl$. Likewise, the value of the function $\chi(S)$ does not depend on $\contextl$ either. We now assume $|\mic| \geq 2$. Let $S \in \calcr$ for some choice of $\contextl = op^{g,m}(\context)$. Let $h \in \mic$ with $h \neq g$ and denote by $J_h$ the incidence relation of $op^{h,m}(\context)$. The fact that $S$ is a mixgen causes $(S \setminus \{h\})^I \neq S^I$, which means that there exists $n \in \hic$ with $n \in (S \setminus \{h\})^I$. Note that $n \notin S^I$. Since $g \in S \setminus \{h\}$, $n \in (S \setminus \{h\})^I$ and $g \in \mic$, we have $n \neq m$ and $n \in S^{J_h}$. Hence, in particular, $S^{J_h} \neq S^I$ and $(S \cup \{g\})^{J_h} = S^{J_h} \setminus \{m\} \neq S^I$, which causes $S \in \calcr$ for the choice $\contextl = op^{h,m}(\context)$.

%$(S \setminus \{h\})^I \cap (\hic \setminus \{m\}) \neq \emptyset$, implying $n \in (S \cup \{h\})^J$ and, in particular, $(S \cup \{h\})^J \neq S^I$ where $J$ denotes derivation in $op^{h,m}(\context)$. 

  For the first equivalence, let $S \in \caln \cup \cala$ and $h \in R$. We need to show first that $(S \setminus R)^I \subseteq S^I \cup \{m\}$. If $S \in \caln$, then both containments follow from Proposition~\ref{carac_not_a_mixgen_inL}. If not, then $S \in \cala$ and $S^I = S^J$. We assume $g \in S$ because, otherwise, $S^J = S^I$ forces $S \cap R = \emptyset$ and then there is nothing to prove. First we prove that $S \setminus \{g\}$ does not strongly avoid $g$. Set $T = S \setminus \{g\}$. A simple calculation shows $S^I = (T \cup \{g\})^J = T^J \cap g^J = (T^I \cup \{m\}) \setminus \{m\} = T^I \setminus \{m\}$ and, consequently, $S^I \cup \{m\} \supseteq T^I = (S \setminus \{g\})^I$, that is, $(S \setminus \{g\})^I \cap \gic \subseteq \{m\}$ and $S \setminus \{g\}$ does not strongly avoid $g$. It must hold as well that $S \cap R \subseteq \{g\}$ because, otherwise, we would have $m \notin (S \setminus \{g\})^I$ and $(S \setminus \{g\})^I = S^I$, contradicting that $S$ is a mixgen. In any case, the containments $S^I \subseteq (S \setminus R)^I \subseteq S^I \cup \{m\}$ hold and, consequently, for arbitrary $h \in R$:
  $$
  \begin{aligned}
    \mbox{$S$ strongly avoids $h$} &\Leftrightarrow S^I \cap \left(\hic \setminus \{m\}\right) \neq \emptyset\\
                                   &\Leftrightarrow \left(S^I \cup \{m\}\right) \cap \left(\hic \setminus \{m\}\right) \neq \emptyset\\
                                   &\Leftrightarrow \left(S \setminus R\right)^I \cap \left(\hic \setminus \{m\}\right) \neq \emptyset\\
                                   &\Leftrightarrow \mbox{$S \setminus R$ strongly avoids $h$}.
  \end{aligned}
  $$
For the converse, suppose that $S \notin \caln \cup \cala$, that is, $S \in \calb \cup \calc$. Set $T = S \cap R$. Clearly, $T \neq \emptyset$. Suppose that $|T| \geq 2$ and take $h \in T$. Then, since $S$ is a mixgen, it follows that $(S \setminus \{h\})^I \neq S^I$, which is equivalent to $(S \setminus \{h\})^I \cap \hic \neq \emptyset$. From $|T| \geq 2$ follows that $m \notin (S \setminus \{h\})^I$, which in turn implies that $S \setminus \{h\}$ strongly avoids $h$. Proposition~\ref{chi_antitono} implies that $S \setminus R$ strongly avoids $h$ as well and, consequently, $\chi(S \setminus R) \neq \chi(S)$. For the case $|T| = 1$, suppose additionally that $S \in \calb$. Then, $T = \{h\}$ with $h \neq g$, since no element in $\calb$ may contain $g$. We claim that $(S \setminus \{h\})^I$ contains properly $S^I \cup \{m\}$. The containment is clear, we have to discard equality: if $(S \setminus \{h\})^I = S^I \cup \{m\}$ then, $(S \setminus \{h\})^J = S^I \cup \{m\}$ and, by Proposition~\ref{carac_not_a_mixgen_inL}, the set $S$ would belong to $\caln$ (and not to $\calb$). Therefore, there exists $n \in (S \setminus \{h\})^I$ with $n \neq m, n \in \hic$. Consequently, $S \setminus \{h\}$ strongly avoids $h$ and, like before, $\chi(S \setminus R) \neq \chi(S)$. Now, suppose that $S \in \calc$. In this case, $T = \{g\}$. Take an element $n \in S^J \setminus S^I$ with $n \neq m$. Since the only object $h \in G$ with $h^J \neq h^I$ is $g$, it holds that $n \in (S \setminus \{g\})^I$ and $n \in \gic$. Hence, in particular, the set $(S \setminus \{g\})^I \cap (\gic \setminus \{m\})$ is non-empty, meaning that $S \setminus \{g\}$ strongly avoids $g$. The second equivalence follows from the first and Proposition~\ref{chi_antitono}.
%Therefore, $S^I \cap ()
%$S^J \supsetneq S^I \setminus \{m\}$. Therefore, there exists $n \in (S \setminus \{g\})^I$ with $n \notin g^I$ and $n \neq m$. Thus, $S \setminus \{g\}$ strongly avoids $g$ and $\chi(S \setminus R) \neq \chi(S)$. 
\end{proof}

Instead of mapping $\caln$ and $\calcr$ ``internally'' into $\cals$, it is possible to map both directly into the set of intents of $\contextl$ via $res \circ \beta \circ J$, where $J$ denotes derivation in $\contextl$ and $\beta$ adds the object $g$, that is, $\beta$ is as defined in Theorem~\ref{teorema1}. The characterization of sets which are mixed generators in $\context$ but not in $\contextl$ implies that $(res(S) \cup \{g\})^J = (S^I \cup \{m\}) \cap g^J = S^I$ whenever $S \in \caln$. Similarly, it is clear that $(res(S) \cup \{g\})^J = (S \setminus R)^I \setminus \{m\}$ for every $S \in \calcr$. In this setting, it is not required anymore that $\cals$ be complete. Moreover, we partition $\cala$ in two, instead of three: we set $\calanr = \calaum \cup \caladois$. Figure~\ref{fig:decomp_tudo} illustrates those facts as well as properties stated in Theorem~\ref{teorema1} and Lemma~\ref{lemma_stability}.

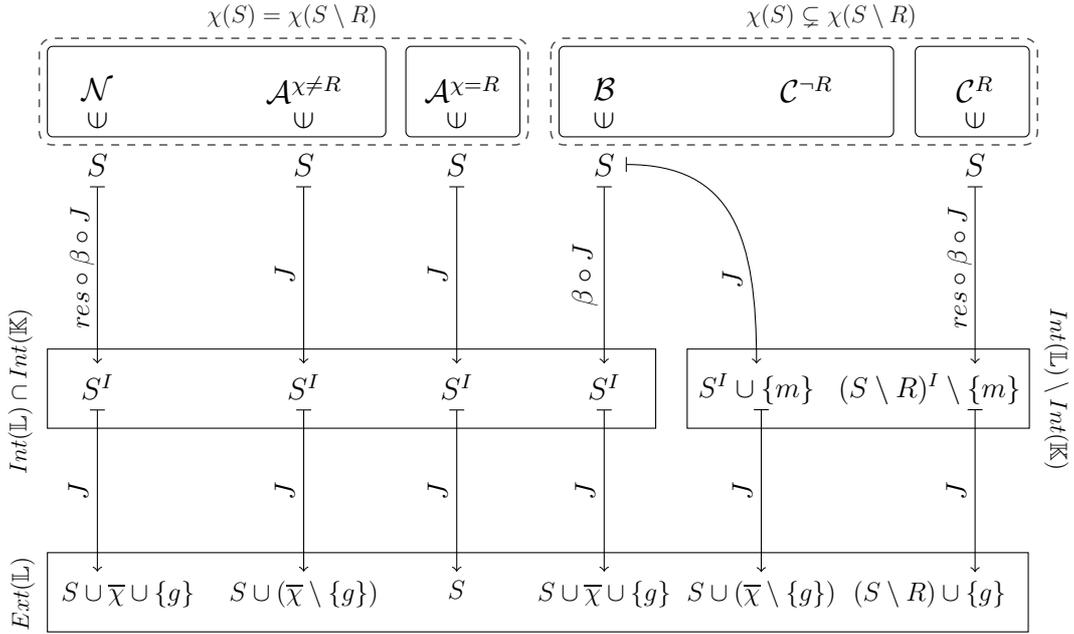
\begin{figure}[htbp]
  \centering
  \begin{tikzpicture}
    [x=1mm,y=1mm,
      classe/.style={rectangle, draw=white!50, inner sep=0pt, minimum width=7mm, minimum height=5mm},
      cup/.style={text=white},
      classegrande/.style={rectangle,draw=red!50, inner sep=0pt, minimum width=18mm, minimum height=5mm},
      classeph/.style={rectangle, draw=white!50, minimum width=1mm},
      classephgrande/.style={rectangle, draw=white!50},
      classein/.style={rectangle,draw=white!100, inner sep=0pt, minimum width=1mm, minimum height=1mm}
    ]
    \scalebox{1}{ % change here to scale. it seems that the figure caption must be corrected manually
    \node[classe] at (0,0) (zero) {};
    \node[classe] (caln) [right=1mm of zero] {$\caln$};
    \node[classe] (calanr) [right=18.5mm of caln] {$\calanr$}; 
    \node[classe] (calar) [right=10.75mm of calanr] {$\calar$}; 
    \node[classe] (calb) [right=10mm of calar] {$\calb$}; 
    \node[classe] (calcnr) [right=19.5mm of calb] {$\calcnr$}; 
    \node[classe] (calcr) [right=15mm of calcnr] {$\calcr$}; 

    \node[classein] (in0) [below=1.5mm of caln, rotate=90, anchor=center] {$\in$};
    \node[classeph] (eldom0) [below=3mm of in0.center] {$S$};
%    \node[classein] (in0) [rotate=90, below=0mm of caln, left=0mm of caln.south] {$\in$};

    \node[classein] (in1) [below=1.5mm of calanr, rotate=90, anchor=center] {$\in$};
    \node[classeph] (eldom1) [below=3mm of in1.center] {$S$};

    \node[classein] (in2) [below=1.5mm of calar, xshift=-0.75mm , rotate=90, anchor=center] {$\in$};
    \node[classeph] (eldom2) [below=4.5mm of calar, xshift=-0.75mm] {$S$};

    \node[classein] (in3) [below=1.5mm of calb, rotate=90, anchor=center] {$\in$};    
    \node[classeph] (eldom3) [below=3mm of in3.center] {$S$};

    \node[classein] (in4) [below=1.5mm of calcr, rotate=90, anchor=center] {$\in$};
    \node[classeph] (eldom4) [below=3mm of in4.center] {$S$};

    % letreiros de intencoes. Esse letreiro eh absoluto e nao deve ser movido daqui
    \node(letreiro1) at (-2,-40) [rotate=90,scale=0.75] {$\intents \contextl \cap \intents \context$};

    \node[classeph] (ph0) [below=23.5mm of eldom0, scale=0.96] {$S^I$};    
    \node[classeph] (ph1) [below=23.5mm of eldom1, scale=0.96] {$S^I$};
    \node[classeph] (ph2) [below=23.5mm of eldom2, scale=0.96] {$S^I$};
    \node[classeph] (ph3) [below=23.5mm of eldom3, scale=0.96] {$S^I$};
    \node[classephgrande] (ph4) [below right=23.5mm and 8.25mm of eldom3, scale=0.9] {$S^I \cup \{m\}$};
    \node[classephgrande] (ph5) [below=23.5mm of eldom4, xshift=-6mm, scale=0.9] {$(S\setminus R)^I \setminus \{m\}$};
     
    \draw[|->] (eldom0.south) -- (ph0.north) node [pos=0.46, xshift=-2mm, rotate=90, scale=0.85] {$res \circ \beta \circ J$};
    \draw[|->] (eldom1.south) -- (ph1.north) node [midway, xshift=-2.5mm, rotate=90, scale=0.9] {$J$};
    \draw[|->] (eldom2.south) -- (ph2.north) node [midway, xshift=-2.5mm, rotate=90, scale=0.9] {$J$};
    \draw[|->] (eldom3.south) -- (ph3.north) node [midway, xshift=-2.5mm, rotate=90, scale=0.9] {$\beta \circ J$};
    \draw[|->] (eldom3.east) .. controls ($(eldom3.center) + (17,0)$) and ($(ph4.center) + (0,20)$) .. (ph4.north) node [pos=0.75, xshift=-2.5mm, rotate=90, scale=0.9] {$J$};
    \draw[|->] (eldom4.south) -- ($(ph5.north) + (6,0)$) node [pos=0.46, xshift=-2mm, rotate=90, scale=0.85] {$res \circ \beta \circ J$};

    \node(ext0) [below=21.5mm of ph0.south, xshift=4mm, scale=0.83] {$S \cup \overline{\chi} \cup \{g\}$};
    \node(ext1) [below=21.5mm of ph1.south, scale=0.83] {$S \cup \left(\overline{\chi} \setminus \{g\}\right)$};
    \node(ext2) [below=21.5mm of ph2.south, scale=0.83] {$S$};
    \node(ext3) [below=21.5mm of ph3.south, scale=0.83] {$S \cup \overline{\chi} \cup \{g\}$};
    \node(ext4) [below=21mm of ph4.south, xshift=0.625mm, scale=0.83] {$S \cup \left(\overline{\chi} \setminus \{g\}\right)$};
    \node(ext5) [below=21mm of ph5.south, xshift=0mm, scale=0.83] {$(S \setminus R) \cup \{g\}$};

    \draw[|->] (ph0.south) -- ( $(ext0.north) + (-4,0)$) node [midway, xshift=-2.5mm, rotate=90, scale=0.9] {$J$};
    \draw[|->] (ph1.south) -- (ext1.north) node [midway, xshift=-2.5mm, rotate=90, scale=0.9] {$J$};
    \draw[|->] (ph2.south) -- (ext2.north) node [midway, xshift=-2.5mm, rotate=90, scale=0.9] {$J$};
    \draw[|->] (ph3.south) -- (ext3.north) node [midway, xshift=-2.5mm, rotate=90, scale=0.9] {$J$};
    \draw[|->] ($(ph4.south) + (0.625,0.6)$) -- (ext4.north) node [midway, xshift=-2.5mm, rotate=90, scale=0.9] {$J$};
    \draw[|->] ($(ph5.south) + (6,0.6)$) -- ($(ext5.north) + (6,0)$) node [midway, xshift=-2.5mm, rotate=90, scale=0.9] {$J$};
    
    % retangulo dos intents na interseccao
    \node (retangulointer) [rectangle, draw, anchor=south west, minimum height=10.5mm, minimum width=80mm] at ($(letreiro1.center) + (3.5,-5)$) {};
    % retangulo dos intents soh de L
    \node (retangulosohl) [rectangle, draw, anchor=south west, minimum height=10.5mm, minimum width=45mm] at ($(retangulointer.south east) + (4,0)$) {};
    % letreiro do retangulo dos intents soh de L
    \node(letreiro2) at ($(retangulosohl.360) + (3.5,0)$) [rotate=-90, scale=0.75] {$\intents \contextl \setminus \intents \context$};

    % retangulo das extensoes
    \node (retanguloext) [rectangle, draw, anchor=south west, minimum height=10.5mm, minimum width=129mm] at ($(retangulointer.south west) + (0,-27.0)$) {};
    %letreiro das extensoes
    \node(letreiro3) [below left=8.7mm and 5mm of letreiro1, rotate=90, scale=0.75] {$\extents \contextl$};

    % retangulos das classes

    % N, Anr
    \node (classena) [rectangle, draw, rounded corners=2pt, anchor=south west, minimum height=12mm, minimum width=44.5mm] at ($(retangulointer.north west) + (0,+28)$) {};

    % Ar
    \node (classear) [rectangle, draw,rounded corners=2pt, anchor=south west, minimum height=12mm, minimum width=15mm] at ($(classena.south east) + (2.5,0)$) {};

    % B, Cnr
    \node (classebcnr) [rectangle, draw, rounded corners=2pt, anchor=south west, minimum height=12mm, minimum width=44mm] at ($(classear.south east) + (5,0)$) {};

    % Cr
    \node (classecr) [rectangle, draw, rounded corners=2pt, anchor=south east, minimum height=12mm, minimum width=15mm] at ($(retangulosohl.north east) + (0,+28)$) {};

    % Chi igual
    \draw [semithick, dashed, black!65,rounded corners=5pt] ($(classena.north west) + (-1,1)$) rectangle ($(classear.south east) + (1,-1)$);
    \node [color=black!90, above left=4.5mm and 5mm of calar, scale=0.75] (chiigual) {$\chi(S) = \chi(S \setminus R)$};

    % Chi diferente
    \draw [semithick, dashed, black!65, rounded corners=5pt] ($(classebcnr.north west) + (-1,1)$) rectangle ($(classecr.south east) + (1,-1)$);
    \node [color=black!90, above left=4.5mm and 3mm of calcr, scale=0.75] (chiigual) {$\chi(S) \subsetneq \chi(S \setminus R)$};
    }
  \end{tikzpicture}
  \caption{Decomposition of a system of mixed generators and corresponding derivations.}
  \label{fig:decomp_tudo}
\end{figure}

We say that a family of $R$-mixed generators $\cals \subseteq \mathcal{P}(G)$ has the \emph{semi-downset property} if the implication $S \in \cals, T \subseteq R \Rightarrow S \setminus T \in \cals$ holds.

\begin{proposition}\label{prop-semidownset}
  For any context $\context$ with finite object set $G$ and an arbitrary $R \subseteq G$ there exists a complete system of $R$-mixed generators which has the semi-downset property.
\end{proposition}
\begin{proof} We prove this constructively. Define some arbitrary total order on $G$. Consider the induced lexicographic order over $\mathcal{P}(G)$ and for every extent $A \in \extents \context$, define $S(A) = \min\{S \subseteq G \mid S \mbox{ is a mixed generator of } A\}$ and $\cals = \{S(A) \mid A \in \extents \context\}$. We show that, for every $S \in \cals$ and every $T \subseteq S \cap R$, the set $S \setminus T$ belongs as well to $\cals$. Set $U = S \setminus T$. By Proposition~\ref{mixgen_obtained_by_restriction}, we have that $U$ is a mixed generator. Suppose, by contradiction, that $U \notin \cals$. Then, there exists a set $V \subseteq G$ with $V < U$, $V \in \cals$ and $U^{II} = V^{II}$ as well as $U^I = V^I$. Note that $T \cap U = \emptyset$ and $V < U$ imply $V \cup T < U \cup T = S$. Then, $(V \cup T)^I = V^I \cap T^I = U^I \cap T^I = (U \cup T)^I = S^I$ which yields $(V \cup T)^{II} = S^{II}$. Since $G$ is finite, we may take a mixed generator of $(V \cup T)^{II}$ which is a subset of $V \cup T$ (possibly $V \cup T$ itself). This contradicts $S \in \cals$.
\end{proof}

It turns out that for every $m \in M$ there exists $g \in \mic$ such that $|\calb| \geq |\calcnr|$, unless the attribute $m$ corresponds to a full column in the formal context.

\begin{theorem}\label{teorema2}
  Let $\context = \GMI$ be a context with finite $G$ and let $m \in M$. Suppose that $\cals$ is a representative system of $\mic$-mixed generators which has the semi-downset property. If $m$ does not correspond to a full column, then there exists an object $g \in \mic$ such that ${|\calb| \geq |\calcnr|}$. 
\end{theorem}
\begin{proof}
Set $R = \mic$ and $n = |R|$. We may of course suppose $n \geq 2$: otherwise $\calcnr = \emptyset$ and there is nothing to prove. Define $\cald^T = \{S \in \calb \cup \calcnr \mid S \cap R = T\}$ for each $T \subseteq R$. Notice that $\cald^{\emptyset} = \cald^R = \emptyset$. For $g \in R$ we further define
$$ \cald^g = \bigcup_{\{g\} \subseteq T \subsetneq R} \cald^T \quad \mbox{ and }\quad \cald^{\neg g} = \bigcup_{\emptyset \subsetneq T \subseteq R \setminus \{g\} } \cald^T.$$
Since $\calb \cup \calcnr = \cald^g \cup \cald^{\neg g}$ and no set in $\calb$ contains $g$ whereas every set in $\calcnr$ does, it actually holds that $\calb = \cald^{\neg g}$ and $\calcnr = \cald^g$. Suppose, by contradiction, that $|\cald^{\neg g}| < |\cald^g|$ for every $g \in R$.
  $$
  \begin{aligned}
    |\cald^{\neg g}| < |\cald^{g}| \Leftrightarrow \sum_{\emptyset \subsetneq T \subseteq R \setminus \{g\}} |\cald^T| < \sum_{\{g\} \subseteq T \subsetneq R} |\cald^T|.
  \end{aligned}
  $$
  Summing for all $g$:
  $$
  \begin{aligned}
    \sum_{g \in R}\, \sum_{\emptyset \subsetneq T \subseteq R \setminus \{g\}} |\cald^T| < \sum_{g \in R}\, \sum_{\{g\} \subseteq T \subsetneq R} |\cald^T|.
  \end{aligned}
  $$
  Consider a proper subset $T$ of $R$ and set $k = |T|$. For each element of $R \setminus T$, there exists a summand $|\cald^T|$ on the left hand side of the inequality above ($n-k$ summands). On the other hand, for each element of $T$, there exists a summand $|\cald^T|$ on the right-hand side ($k$ summands). We calculate a balance: if $k \leq \lfloor n/ 2 \rfloor$, the summand $|\cald^T|$ contributes exactly $(n - 2k)$ times to the left-hand side. If $k \geq \lfloor n/2 \rfloor + 1$, the summand $|\cald^T|$ will appear exactly $(2k-n)$ times on the right-hand side. We have, therefore
\begin{equation}\label{eq1}
  \sum_{k=1}^{\lfloor n/2 \rfloor} (n - 2k) \cdot \sum_{\substack{T \subseteq R\\|T| = k}} |\cald^T| < \sum_{k = \lfloor n/2 \rfloor + 1}^{n-1} (2k-n) \cdot \sum_{\substack{T \subseteq R\\ |T| = k}} |\cald^T| \\
  = \sum_{k = \lceil n/2 \rceil}^{n-1} (2k-n) \cdot \sum_{\substack{T \subseteq R\\ |T| = k}} |\cald^T|\\
\end{equation}
\emph{(the equality above is indeed valid: for odd $n$, it is very clear. For even $n$, notice that $2k-n$ equals zero for $k=n/2$).}

On the other hand, consider the $k$-element subsets of $R$ for each $1 \leq k \leq n-1$, $\calr^k = \{T \subseteq R \mid |T| = k\}$. The semi-downset property of $\cals$ assures the existence of an injection from $\cald^T$ into $\cald^U$ whenever $T \supseteq U$, given by the restriction $S \mapsto S \setminus (T \setminus U)$. Now, for each $i$ with $1 \leq i \leq \lfloor n/2 \rfloor$, take a bijection $ \phi^i: \calr^{n-i} \rightarrow \calr^i$ such that $\phi^i(T) \subseteq T$ for every $T \in \calr^{n-i}$ (such bijections exist by an elementary application of Hall's theorem). Now, let $k \in \{1,2,\ldots, \lfloor n/2 \rfloor\}$, so that $n-2k \geq 0$. Then:
$$
\begin{aligned}
  \sum_{T \in \calr^{n-k}} |\cald^T| \leq \sum_{T \in \calr^{n-k}} |\cald^{\phi^k(T)}| 
  &\Leftrightarrow (n-2k) \cdot \sum_{T \in \calr^{n-k}} |\cald^T| \leq (n-2k) \cdot \sum_{T \in \calr^{n-k}} |\cald^{\phi^k(T)}|\\
  &\Rightarrow \sum_{k=1}^{\lfloor n/2 \rfloor} (n - 2k) \sum_{T \in \calr^{n-k}} |\cald^T| \leq \sum_{k=1}^{\lfloor n/2 \rfloor} (n - 2k) \sum_{T \in \calr^{n-k}} |\cald^{\phi^k(T)}| \\
  &\Leftrightarrow \sum_{k=n - \lfloor n/2 \rfloor}^{n-1} (2k - n) \sum_{T \in \calr^{k}} |\cald^T| \leq \sum_{k=1}^{\lfloor n/2 \rfloor} (n - 2k) \sum_{T \in \calr^{n-k}} |\cald^{\phi^k(T)}| \\
  &\Leftrightarrow \sum_{k=\lceil n/2 \rceil}^{n-1} (2k-n) \sum_{\substack{T \subseteq R\\|T| = k}} |\cald^T| \leq \sum_{k=1}^{\lfloor n/2 \rfloor} (n-2k)\sum_{\substack{T \subseteq R\\|T| = k}} |\cald^T|,
\end{aligned}
$$
which contradicts~\eqref{eq1}.
\end{proof}

As a consequence, the initially posed question gets answered in the affirmative:

\begin{corollary}\label{coro-answer}
  Let $\context = \GMI$ be a context with finite $G$ and let $m \in M$ be an attribute which does not correspond to a full column. Then, there exists an object $g \in \mic$ such that $|\B(\contextl)| \geq |\B(\context)|$, where $\contextl = \opgmk$.
\end{corollary}
\begin{proof}
  One takes a complete system of $\mic$-mixgens with the semi-downset property provided by Proposition~\ref{prop-semidownset} and applies Theorem~\ref{teorema2} in order to fulfill the sufficient condition guaranteed by Theorem~\ref{teorema1}.
\end{proof}

\section{Application}\label{section:app}

Let $a,b,c \in \mathbb{N}$. The contranominal scale of size $j$ will be denoted $\CN(j)$. We define $\calk_{c}^{a,b}$ as the set of all formal contexts having exactly $a$ objects, $b$ attributes and having no contranominal scale of size $c$ as a subcontext. We define $\cont(\context)$ to be the \emph{contrast} of $\context$, that is, the size of the largest contranominal scale (number of objects) that can be found as subcontext of $\context$. We say that $\CN(j)$ is a \emph{contranominal-summand} of $\context$ if $\context = \mathbb{L} + \CN(j)$. The largest natural number $j$ (possibly zero) such that $\CN(j)$ is a contranominal-summand of $\context$ will be denoted by $\cont^*(\context)$. Obviously, $\cont^*$ is a lower bound of $\cont$. The \emph{noncontranominal kernel} of $\context$ is $\contextl$ if $\context = \CN(j) + \contextl$ with $j = \cont^*(\context)$. It is unique\footnote{The noncontranominal kernel corresponds to the connected components of the hypergraph $\{\hic \mid h \in G\}$ which are not singletons.}. Notice that $\contextl$ may have no objects and/or no attributes. More precisely, if $\context$ is the apposition (subposition) of a contranominal scale and a full context, then $\contextl$ will have no objects (attributes). In particular, if $\context$ is a contranominal scale, its noncontranominal kernel $\contextl$ will not have any objects or attributes. Lastly, we define $ex(\calk_{c}^{a,b})$ to be the set of contexts in $\calk_{c}^{a,b}$ which have the maximum number of concepts (among all contexts in $\calk_{c}^{a,b}$).

\begin{theorem}
Amongst the contexts with $n$ objects, $m$ attributes and no contrano\-minal-scale of size $c \leq \min\{n,m\} + 1$, there exists one context with maximum number of concepts which has a contrano\-minal-scale of size $c-1$ as a subcontext.
%  Let $a,b,c \in \mathbb{N}$ with $c \leq \min\{a,b\} + 1$. There exists a formal context $\context \in ex(\calk_{c}^{a,b})$ such that $\cont(\context) = c-1$.
\end{theorem}
\begin{proof} 
  For an arbitrary context $\context$, let $\mathbb{L}$ be its noncontranominal kernel. We define the following non-deterministic operation $nop(\mathbb{\context})$.
$$  
  \begin{aligned}
    nop(\context) = 
    \begin{cases}
      \context, \\
      \mbox{\quad \textbf{if} $\mathbb{L}$ has no objects or no attributes }\\
      op^{g,m}(\mathbb{K}), \mbox{ where $(g,m)$ is some non-incidence of $\mathbb{L}$, with $g$ as in Corollary~\ref{coro-answer}}\\
      \mbox{\quad \textbf{if} $\mathbb{L}$ has objects and attributes and some non-incidence} \\
      \mathbb{K} - (g,m),  \mbox{ where $(g,m)$ is any object/attribute pair belonging to $\mathbb{L}$,}\\ 
      \mbox{and $\context - (g,m)$ means the context $\context$ without the incidence $(g,m)$}\\
      \mbox{\quad \textbf{if} $\mathbb{L}$ has objects and attributes but no non-incidence ($\mathbb{L}$ is a full context)}\\
    \end{cases} 
  \end{aligned}
$$ 
  It should be clear that $nop(\context)$ has always at least as many concepts as $\context$. Moreover, $nop(\context)$ is either $\context$ itself or $nop$ increases $\cont^*$ by exactly one. We claim that 
$$\cont(\context) \leq \cont(nop(\context)) \leq \cont(\context) + 1.$$
Indeed, consider a maximum contranominal scale $\context_1 = \CN(\cont) \leq \context$. We may suppose $nop(\context) \neq \context$, thus, a pair $(g,m)$ was chosen by the operation $nop$. If both $g$ and $m$ are inside $\context_1$, then clearly $\context_1 \leq nop(\context)$. If neither $g$ nor $m$ are inside $\context_1$, then obviously $\CN(\cont+1) \leq nop(\context)$. If only one among $g$ and $m$ are inside $\context_1$, then there exists a contranominal scale of size $\cont - 1$ with neither $g$ nor $m$. Hence, $g$ and $m$ will, together, belong to a $\CN(\cont)$ in $nop(\context)$. This proves the lower bound $\cont(nop(\context)) \geq \cont(\context)$. The upper bound is clear.

  Let $\context \in ex(\calk_{c}^{a,b})$. Consider the nondeterministic sequence
$$(\context, nop(\context), nop(nop(\context)), nop(nop(nop(\context))), ... ) $$
The last context in the sequence above that still belongs to $\calk_c^{a,b}$ must have $\CN(c-1)$ as a subcontext. Its extremality follows from the extremality of $\context$. 
\end{proof}

%It is natural to consider the stronger version of the above. We do not know any proof or counterexamples.
%\begin{question}
%  Let $a,b,c \in \mathbb{N}$ with $c \leq \min\{a,b\} + 1$. Does every formal context $\context \in ex(\calk_{c}^{a,b})$ satisfy $\cont(\context) = c-1$?
%\end{question}
A natural way of attacking this question aiming to prove the full conjecture would be trying to prove that a ``smart'' choice of $g$ done by Corollary~\ref{coro-answer} actually achieves $|\B(\opgmk)| > |\B(\context)|$. The context depicted in Figure~\ref{fig_resistente}, though, is ``resistant'' to such idea. Actually, for \textbf{any} choice of $m$ \textbf{and} $g$, the resulting context $\opgmk$ is such that $|\B(\opgmk)| = |\B(\context)| = 22$ ({\sc ConExp} or similar comes in handy).
\begin{figure}[htbp]
  \begin{center}
  \begin{cxt}                                                               
    \cxtName{}                                                   
    \att{$a$}                                                                  
    \att{$b$}                                                                  
    \att{$c$}                                                        
    \att{$d$}
    \att{$e$}
    \att{$f$}
    \obj{..xxxx}{1}                                                          
    \obj{xx..xx}{2}                                                      
    \obj{xxxx..}{3}                                                         
    \obj{.x.x.x}{4}                                                          
    \obj{x.x.x.}{5}                                                    
  \end{cxt}     
  \end{center}
  \caption{A context for which one operation $\opgm$ is not able to increase the number of concepts, independent of the choice of $g$ and $m$.}
  \label{fig_resistente}
\end{figure}

%\section{Appendix}

%The following are quite well-known exercises of graph theory. They are applications of Hall's theorem.

%\begin{proposition}
%  Every regular bipartite graph has a perfect matching. Actually, every $k$-regular bipartite graph has $k$ pairwise disjoint perfect matchings. 
%\end{proposition}
%\begin{proof}[idea]The proof usually goes by induction on $k$ and checking that Hall's condition is fulfilled.
%\end{proof}

%\begin{proposition}
%  Let $R$ be a finite set with $n$ elements and let $0 \leq k \leq \lfloor n/2 \rfloor$. There exists a bijection between its $k$-element sets and $(n-k)$-element sets that associates supersets to each set: $S \mapsto T$ with $S \subseteq T$.
%\end{proposition}
%\begin{proof}
%  A subset having exactly $x$ elements will be called a \emph{$x$-subset}. Consider the ``obvious'' (balanced) bipartite graph, given by $k$-subsets in one class and $(n-k)$-subsets in the other class. Adjacencies in this graph are given by set inclusion $\subseteq$. 

%  Now, every $k$-subset is contained in ${n-k \choose n-2k}$ $(n-k)$-subsets and every $(n-k)$-subset contains ${n-k \choose k}$ $k$-subsets. Since  ${n-k \choose n-2k} =  {n-k \choose n-k-k} = {n-k \choose k}$, the graph is regular and has therefore a perfect matching.
%\end{proof}

%\bibliography{mixgens}

\end{document}